\documentclass[a4paper,11pt]{amsart}
\usepackage[ansinew]{inputenc}

\usepackage{amssymb}
\usepackage{amsthm}
\usepackage{amsmath}

\newenvironment{narrow}[2]{
 \begin{list}{}{% 
  \setlength{\topsep}{0pt}% 
  \setlength{\leftmargin}{#1}% 
  \setlength{\rightmargin}{#2}% 
  \setlength{\listparindent}{\parindent}% 
  \setlength{\itemindent}{\parindent}% 
  \setlength{\parsep}{\parskip}% 
 }%
\item[]}{\end{list}} 

\usepackage{graphicx}
\newcommand{\imageplain}[2] % Image scaled
{
      {\includegraphics[scale=#1]{#2.jpeg}}
}
\newcommand{\image}[2] %Image scaled centered
{
      \begin{center} %Image
	\includegraphics[scale=#1]{#2.jpeg}
      \end{center}
}

\usepackage{aliascnt}
\usepackage{hyperref}
\hypersetup{ 
    colorlinks=false,
    pdfborder={0 0 0},}
\newcommand{\nref}[1]{\hyperref[#1]{\ref*{#1}}}

\newcommand{\Ker}{{\rm Ker}}
\newcommand{\Hom}{{\rm Hom}}
\renewcommand{\Im}{{\rm Im}}
\newcommand{\Aut}{{\rm Aut}}
\newcommand{\ad}[2]{{\rm ad}\left(#1\right)\left(#2\right)}
\newcommand{\adn}[3]{{\rm ad}^{#1}\left(#2\right)\left(#3\right)}
\renewcommand{\mod}{\;{\rm mod}\;}

\renewcommand{\O}{\mathcal{O}} 	% Matrix Function Algebra
\newcommand{\B}{\mathcal{B}} 	% Nichols Algebra
\newcommand{\g}{\mathfrak{g}} 	% arbitrary Lie algebra
\newcommand{\Z}{\mathbb{Z}}  	% Integers
\newcommand{\C}{\mathbb{C}}  	% Complex
\renewcommand{\k}{\Bbbk}  	% arbitrary Field
\newcommand{\F}{\mathbb{F}}  	% finite field
\newcommand{\D}{\mathbb{D}}  	% Dihedral
\newcommand{\Q}{\mathbb{Q}}  	% Quaternion
\renewcommand{\S}{\mathbb{S}}  	% Symmetric
\newcommand{\A}{\mathbb{A}}  	% Alternating

\renewcommand{\H}{\mathcal{H}} %\Griess algebra

\newcommand{\ydm}{Yetter-Drinfel'd module }
\newcommand{\ydms}{Yetter-Drinfel'd modules }
\newcommand{\ydmP}{Yetter-Drinfel'd module}
\newcommand{\ydmsP}{Yetter-Drinfel'd modules}

\theoremstyle{plain}
\newtheorem{theorem}{Theorem}[section]

\newtheorem{corollary}[theorem]{Corollary}

\newtheorem{definition}[theorem]{Definition}

\newtheorem{example}[theorem]{Example}
\newtheorem*{exampleX}{Example}

\newtheorem{lemma}[theorem]{Lemma}

\newtheorem{remark}[theorem]{Remark}

\begin{document}

  \thispagestyle{empty}
  \vspace*{-2cm}
   \begin{flushright}
     ZMP-HH / 13-4 \\
     Hamburger Beitr{\"a}ge zur Mathematik Nr. 471 \\
   \end{flushright}
  \vspace{1.5cm}
  \begin{center}
    \Large 
  New Large-Rank Nichols Algebras Over \\ Nonabelian Groups 
  With Commutator Subgroup $\Z_2$\\
  \vspace{.3cm}
    \normalsize
  ({\bf long version}, a shortend version appeared in Journal of Algebra 419)
  \end{center}
  \vspace{0.5cm}
  \begin{center}	
    Simon Lentner \\
    Algebra and Number Theory (AZ), 
    Universit{\"a}t Hamburg,\\
    Bundesstra{\ss}e 55, D-20146 Hamburg \\
  \end{center}
  \vspace{0.5cm}

\begin{abstract}
In this article, we explicitly construct new finite-dimensional,
indecomposable Nichols algebras with Dynkin diagrams of type
$A_n,C_n,$ $D_n,E_{6,7,8},F_4$ over any group $G$ with commutator subgroup
isomorphic to $\Z_2$.
The construction is generic in the sense that the type just depends on the rank
and center of $G$, and thus positively answers for all groups of this class a
question raised by Susan Montgomery in 1995 \cite{Mont95}\cite{AS02}.\\

Our construction uses the new notion of a covering Nichols algebra as a special
case of a covering Hopf algebra \cite{Len12} and produces non-faithful Nichols
algebras. We give faithful examples of Doi twists for type
$A_3,C_3,D_4,C_4,F_4$ over several nonabelian groups of order $16$ and $32$.
These are hence the first known examples of nondiagonal, finite-dimensional,
indecomposable Nichols algebras of rank $>2$ over nonabelian groups.\\
\end{abstract}

\makeatletter
\@setabstract
\makeatother

\newpage

\tableofcontents
\newpage

\section{Introduction}

The Nichols algebra $\B(M)$ of a \ydm $M$ over a group $\Gamma$ is a quotient
of the tensor algebra $T(M)$. It has a natural structure of a Hopf algebra in
a braided category satisfying a certain universal property.
Finite-dimensional Nichols algebras
arise naturally e.g. as quantum Borel part in the classification of
finite-dimensional pointed Hopf algebras \cite{AS10}, such as the small
quantum groups $u_q(\g)$.
Heckenberger classified all finite-dimensional Nichols algebras for
$\Gamma$ abelian \cite{Heck09}. For Nichols algebras over arbitrary semisimple
Yetter-Drinfeld modules, Andruskiewitsch, Heckenberger and
Schneider defined a Weyl groupoid and a generalized root system in
\cite{AHS10}, and further developed it in \cite{HS10a}.

However, the existence of a finite-dimensional Nichols algebra over
a nonabelian group still seems to be a rather rare and difficult phenomenon.
The first examples were discovered by Milinski and Schneider in \cite{MS00} over
Coxeter groups and are of rank $1$ except for one of rank $2$ over $\D_4$. Here
\emph{rank} refers to the number of irreducible summands in the underlying
\ydmP, and equivalently to the rank of the Weyl groupoid and the root system.
Andruskiewitsch, Gra{\~n}a, Heckenberger, Lochmann and Vendramin have in
\cite{AG03}\cite{GHV11}\cite{HLV12} constructed several large finite-dimensional
Nichols algebras of rank 1. On the other hand, strong conditions have been
developed to rule out the existence of finite-dimensional Nichols-algebras over
many groups such as higher alternating groups and most sporadic groups, see e.g.
\cite{AFGV11b}.\\
Moreover, during the work on this article, Schneider, Heckenberger and
Vendramin have in series of papers completed the classification of Nichols
algebras of rank $2$, see \cite{HV13}, by narrowing down the possibilities
using the root system theory and constructing the remaining by hand. So
far, no examples of higher rank over nonabelian groups have been constructed.\footnote{Added 04/2015: Recently Heckenberger and Vendramin have in \cite{HV14} completed the classification of all Nichols algebras of rank $>1$. In rank $>3$ and characteristic $0$ the Nichols algebras constructed in the present article turn out to be all; this is very furtunate, because our construction gives very much control over the Nichols algebra (generators and relations, cohomology etc.). There is a single exceptional example in rank $3$ and in characteristic $3$ there is an additional series $B_n$ which seemingly can be constructed with the approach in the present article from an additional $B_n$ series of diagonal Nichols algebras with braiding $\pm 1$ in characteristic $3$.}\\

In this article, we explicitly construct finite-dimensional
indecomposable Nichols algebras with root systems of type
$A_n,\;C_n,\;D_n,\;E_{6,7,8},\;F_4$ and hence arbitrary rank over
nonabelian groups $G$
that are central stem-extensions of an abelian group $\Gamma$, i.e.:
$$\Sigma^*=\Z_2\rightarrow G\rightarrow \Gamma\qquad \Sigma^*\subset
[G,G]\cap Z(G)$$
As a side remark, we mention that the construction is a special case of our new
notion of a covering Hopf algebra \cite{Len12}, applied to the bosonization
of known finite-dimensional Nichols algebra over the abelian group $\Gamma$. The
covering construction itself does not depend on $\Gamma$ being abelian: For
example, in \cite{Len12} we have constructed a covering Nichols algebra of
dimension $24^2$ over $G=GL_2(\F_3)$, which is a $\Z_2$-stem-extension of
$\Gamma=\S_4$ and an open case called ${\bf C_4}$ in the list \cite{FGV07}.\\

More concretely, we proceed in Section \ref{sec_twistedSymmetry} as follows:
Suppose $M=\bigoplus_{i\in I}M_i$ is a semisimple \ydm over the abelian finite
group $\Gamma$ with simple $1$-dimensional summands $M_i,\;i\in I$, diagonal
braiding matrix $q_{ij}$ and known finite-dimensional Nichols algebra $\B(M)$.
Furthermore, suppose that a finite abelian group $\Sigma$ acts on the
vector space $M$, such that the $\Gamma$-graduation as well as the the
self-braiding operators $c_{M_iM_i}=q_{ii}$ and the monodromy operators
$c_{M_iM_j}c_{M_iM_j}=q_{ij}q_{ji}$ are preserved. We usually assume
the $\Sigma$-action induced from a permutation action on $I$. In this case the
assumed compatibility with the braiding lets $\Sigma$ act on the $q$-diagram
and Dynkin-diagram of $M$ (having nodes $i\in I$) by graph automorphisms.

However, $\Sigma$ does not act on $M$ by \ydmP-automorphisms: The
braiding matrix $q_{ij}$ itself is generally not
preserved, but is supposed to be modified under the action of each $p\in\Sigma$
as prescribed by a bimultiplicative form $\langle\bar{g_i},\bar{g_j}\rangle_p$
with respect to the $\Gamma$-graduation $\bar{g_i},\bar{g_j}\in\Gamma$ of
$M_i,M_j$. \\

The bimultiplicative forms $\langle\rangle_p$ will usually be induced
from a group-2-cocycle $\sigma\in Z^2(\Gamma,\Sigma^*)$ via
$\langle\bar{g},\bar{h}\rangle_p:=\sigma(\bar{g},\bar{h})(p)\sigma^{-1}
(\bar{h},\bar{g})(p)$. We call such an action a twisted symmetry
action of $\Sigma$ on $M$ with respect to the the 2-cocycle $\sigma$.\\

With these notions we construct in Section \nref{sec_Construction} a
covering \ydm $\tilde{M}$ over a stem-extension $\Sigma^*\rightarrow
G\rightarrow \Gamma$ with $\Gamma$ abelian as follows: We start with a
$\Gamma$-\ydm $M$ and a twisted permutation symmetry action of $\Sigma$ on
$M$ (i.e. induced from $\Sigma$ permuting $I$) with respect to a group-2-cocycle
$\sigma$ representing the given stem-extension.
We decompose $M$ into simultaneous eigenspaces $M^{[\lambda]}$ for
eigenvalues $\lambda\in\Sigma^*$ of the
twisted symmetry action of $\Sigma$ and use this $\Sigma^*$-graduation to
refine the $\Gamma$-graduation
on $M$ to a $G$-graduation. Note that formerly
$\Gamma$-homogeneous elements in $M$ are usually
not $G$-homogeneous. By pulling back also the
$\Gamma$-action on $M$ to a $G$-action, we obtain a covering \ydm $\tilde{M}$
over the nonabelian group $G$. As a braided
vector space, $\tilde{M}$ is isomorphic to $M$. The Nichols
algebra  $\B(\tilde{M})$ of the covering \ydm $\tilde{M}$ is called covering
Nichols algebra and is isomorphic to $\B(M)$ as an algebra.\\

When we apply this construction in case $\Gamma$ abelian to a semisimple
Yetter-Drinfel'd module $M=\bigoplus_{i\in I}M_i$ with $\Sigma$ acting by
twisted permutation symmetries, then the different
irreducible $\Gamma$-\ydms $M_i$ laying on an orbit
of the twisted symmetry $\Sigma$ become a single irreducible $G$-\ydm with
increased dimension. This changes the Cartan matrix and hence the Dynkin
diagram (see \cite{HS10a} Definition 6.4) of $\tilde{M}$ compared with $M$ as
described in Theorem \nref{thm_Fold}:
Nodes of the Dynkin diagram of $M$ 
in a $\Sigma$-orbit give rise to a single nodes of the Dynkin diagram
of $\tilde{M}$. The root system is reduced to the subsystem fixed
by $\Sigma$ acting on the Dynkin diagram of $M$ by graph automorphisms. This
behaviour is classically known as diagram folding of a Lie algebra by an outer
automorphism (for a purely root system approach to folding see e.g.
\cite{Ginz06} p. 47). 

\begin{exampleX}(Section \nref{sec_RamifiedEF})
  There exists a $6$-dimensional \ydm $M$ over $\Gamma=\Z_2^4$ with
  $\Gamma$-homogeneous components of dimension $1,1,2,2$ which is the sum of
  $6$ simple 1-dimensional \ydms $M_i$. The Cartan
  matrix of $M$ is as the semisimple Lie algebra $E_6$ and the Nichols algebra
  $\B(M)$ has dimension  $2^{36}$. Moreover, $M$ admits an  action of
  $\Sigma=\Z_2$  by twisted  symmetries, corresponding to a diagram 
  automorphism of the $E_6$ root system.\\

  The covering  \ydm $\tilde{M}$ over the 
  stem-extension  $G=\Z_2^2\times \D_4$ is the sum of $4$ simple \ydms
  $\tilde{M}_i$ of dimensions $1,1,2,2$. The covering Nichols algebra
  $\B(\tilde{M})$ is indecomposable, has also dimension $2^{36}$ and its
  Dynkin diagram is $F_4$. This corresponds to the inclusion of the semisimple
  Lie algebras $F_4\subset E_6$ as fixed points of the outer Lie algebra
  automorphism of $E_6$.
\end{exampleX}

This example is visualized as follows:
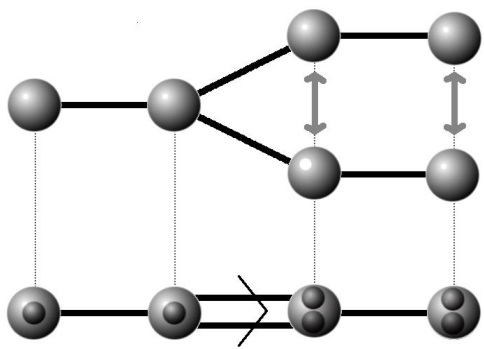
\image{0.26}{Title}

For twisted permutation symmetries of prime order $\Sigma=\Z_p$ we use the
suggestive terms inert/split for $G$-nodes $\tilde{M}_i$ of dimension $1$/$p$.
They correspond to
$\Sigma$-orbits of length $1$/$p$ of $\Gamma$-nodes $M_i$ and by Lemma
\nref{lm_SplitInert} to central/noncentral $G$-graduation (conjugacy class!).
We
also use the terms inert / ramified / split for $G$-edges between
$G$-nodes $\tilde{M}_i,\tilde{M}_j$ that are inert,inert / inert,split /
split,split. They correspond to edges between between $\Gamma$-nodes
in different $\Sigma$-orbits of length $1,1$ / $1,p$ / $p,p$ and by Theorem
\nref{thm_Fold} to $G$-edges of type $A_2$ / $B_2$ / $A_2$. All
these cases are shown in the preceding example for $p=2$.\\

%\enlargethispage{2\baselineskip}
In order to summarize combinatorial considerations, we introduce in Section
\ref{sec_SymplecticRootsystems} the notion of a symplectic root system. A
symplectic root system for a given Cartan matrix
resp. Dynkin diagram is a decoration of the diagram nodes by values in a finite
symplectic vector space, such that the Weyl group acts as symplectic isometries.
Note that such structures and especially the group of isometries have already
been studied under the name ``vanishing lattices'' in \cite{Chm82}, \cite{Jan83}
in the context of singularity theory. The existence of a symplectic root system
encodes nontrivial necessary
conditions on the existence of a covering Nichols algebra: By \cite{HS10a} Prop.
8.1., nodes in the diagram have to be connected if the $G$-decorations
are noncommuting.
Meanwhile we have in \cite{Len13b} completely classified symplectic root
systems for arbitrary given graphs over the field $\F_2$, which is the case
relevant to this article. This technical result determines for a given Dynkin
diagram the possible size of rank and center of a nonabelian group $G$ which
realizes the diagram as a finite-dimensional covering Nichols algebra. For
example, to realize the diagram $D_n$, the group $G$ needs to have a larger
center than for other diagrams, while groups $G$ with even larger center can
only support disconnected diagrams.\\

Note that in \cite{Len12} Theorem 6.1 we also
checked for possible covering Nichols algebras for other primes $\Sigma=\Z_p$
and found that the ones given in the present paper of Cartan type and
$\Sigma=\Z_2$ are indeed the only admissible choices. However, there are
additional diagrams not corresponding to semisimple Lie algebras that lead to
new {\bf de}-composable covering Nichols algebras with $\Sigma=\Z_2$ and even
one $D_4\rightarrow G_2$ for $\Sigma=\Z_3$.\\

The application of the Construction Theorem \nref{thm_Construction} proceeds
case-by-case, depending on the assumed symmetric Dynkin
diagram of $\B(M)$ and a respective symplectic root
system, whose existence imposes restrictions on rank and center of $G$. For the
constructions we asumme $G$ fulfills an additional technical condition
($2$-saturated) to simplify statements involving even-order generating sets. For
groups where this is not the case, we may construct Nichols algebras that are
disconnected and/or contain a connected component solely over an abelian group.
This is discussed in Section \nref{sec_Disconnected}, especially Lemma
\nref{lm_Disconnected} and Corollary \nref{cor_Montgomery}. For each connected
Dynkin diagram we compactly describe dimension, root system and Hilbert series
of the newly constructed covering Nichols algebra $\B(\tilde{M})$.
\begin{itemize}
 \item In Section \nref{sec_Unramified} we treat the generic, unramified case:
    Given a simply-laced Dynkin diagram $X_n$ of type ADE, we use the symplectic
    root system to define a $\Gamma$-\ydm $M=N\oplus N_{\sigma}$ with
    finite-dimensional Nichols algebra $\B(M)$, where $N\not\cong N_\sigma$   
    each have the given Dynkin diagram  $X_n$ and the braiding matrices only
    contain entries $\pm 1$. The crucial aspect of the symplectic root system is
    that it ensures the overall Dynkin diagram of $M$ to be a disconnected
    union $X_n\times X_n$, which corresponds to
    $c_{NN_{\sigma}}c_{N_{\sigma}N}=id$, and hence again a finite-dimensional
    Nichols algebra. 

    Then,
    interchanging $N,N_{\sigma}$ gives by construction an action of
    $\Sigma=\Z_2$ by twisted symmetries on $M$.
    Note that the well-known example of an indecomposable Nichols algebra over
    $G=\D_4$ (see \cite{MS00} resp. example Section \nref{sec_D4} in
    this article) is our model for this case and corresponds to the diagram
    $X_n=A_2$. We also give an example
    $X_n=A_4$ in Section \nref{sec_exampleA4}.
 \item In Section \nref{sec_RamifiedEF} we construct the exceptional
    example visualized above, where the twisted symmetry acts on a single
    $\Gamma$-\ydm $M$ of type $E_6$. The covering Nichols algebra
    $\B(\tilde{M})$ over $G$ has Dynkin diagram
    $F_4$. Thereby ramified edges appear, which connect simple $G$-\ydms of
    different dimension. We use a symplectic root system to construct the split
    part of $M$ with diagram $A_2\times A_2$ and use an ad-hoc continuation by
    two inert nodes to $E_6$.
 \item In Section \nref{sec_RamifiedAB} we construct an infinite
    family of ramified covering Nichols algebras of type $A_{2n-1}\rightarrow
    C_n$ similar to the previous case $E_6\mapsto F_4$. Again we use a
    symplectic root system for the split part of the diagram $A_{n-1}\times
    A_{n-1}$ and an explicit continuation by one inert node to $A_{2n-1}$.
 \item In Section \nref{sec_Disconnected} we describe how to construct Nichols
    algebras $\B(\tilde{M})$ with disconnected Dynkin diagrams and prove in
    particular, that any group $G$ with $[G,G]\cong\Z_2$ (regardless of
    the order) admits at least one finite-dimensional
    indecomposable Nichols algebra.
\end{itemize}

We summarize the properties of the constructed Nichols algebras with connected
Dynkin diagram in the following table, where the first and second column give
necessary and sufficient conditions on the group $G$. For presenting the main
result, we also need an additional technical assumption on $G$ ($2$-saturated,
Definition \nref{def_Saturated}) that gives us control over the size of
even-order generating systems in a group.
We use the conventions
$(n\;\mbox{mod}\;2)\in\{0,1\}$
and denote by $\Phi^+(X_n)$ a fixed set of positive roots in a root system $X_n$
(the Nichols algebra dimensions are $2$-powers, because all self-braidings
are $q_{\alpha\alpha}=-1$).\\

\begin{center}
   \begin{narrow}{-0.8cm}{0cm}
      \begin{tabular}{ll|lll}
	$dim_{\F_2}(G/G^2)\quad$
	  & $dim_{\F_2}(Z(G)/G^2)\quad$ & 
	    Dynkin-D. of $\tilde{M}$ $\quad$
	  & $\dim\left(\B({M})\right)$ & $=\dim\left(\B(\tilde{M})\right)$\\
	\hline
	 $n$ & $n \mod 2$ & $A_{n\geq 2}$ &
	  $2^{|\Phi^+(A_n\times A_n)|}$ & $=2^{n(n+1)}$\\
	 $n=6,7,8$ &  $n \mod 2$ & $E_{6,7,8}$ &
	  $2^{|\Phi^+(E_n\times E_n)|}$ & $=2^{72},2^{126}, 2^{240}$\\
	 $n$ & $2-(n \mod 2)$ & $D_{n\geq 4}$ &
	  $2^{|\Phi^+(D_n\times D_n)|}$ & $=2^{2n(n-1)}$\\
	 $n=4$ & $2$ & $F_4\quad$ &
	    $2^{|\Phi^+(E_6)|}$ & $=2^{36}$\\
         $n$ & $2-(n \mod 2)$ & $C_{n\geq 3}\quad$ &
	    $2^{|\Phi^+(A_{2n-1})|}$ & $=2^{n(2n-1)}$\\
      \end{tabular}
  \end{narrow}
\end{center}~\\

Especially, in Corollary \nref{cor_Montgomery} we find indecomposable Nichols
algebras (possibly with disconnected Dynkin diagram) over all groups $G$, that
are $\Z_2$-stem-extensions of an abelian group $\Gamma$ and thus positively
answer for such groups a respective question
raised by Susan Montgomery in \cite{Mont95} for pointed Hopf algebras by
providing the bosonizations $H=\k[G]\#\B(\tilde{M})$. See \cite{AS02} Question
3.17 for the Nichols algebra formulation.\\

By construction, a Nichols algebra $\B(\tilde{M})$ obtained this way is
non-faithful, diagonal and as an algebra isomorphic to the corresponding Nichols
algebra $\B(M)$ over the abelian $\Gamma$. However, the knowledge
of $H^2(G,\k^\times)$ together with Matsumoto's spectral sequence often allows
to obtain Doi twists that are truly new faithful, non-diagonal,
finite-dimensional, indecomposable Nichols Algebras. We give explicit examples
of type $A_2,\;A_3,\;C_3,\;D_4,\;C_4,\;F_4$ over nonabelian groups of
order $16$ and $32$ in Section \nref{sec_NondiagonalExamples}.

\section{Preliminaries}

Throughout this article we suppose $\k=\C$, all groups are finite and all
vector spaces finite-dimensional. The field with $p$ elements is denoted by
$\F_p$. The dihedral, quaternion, symmetric and alternating groups are denoted
by $\D_4,\Q_8,\S_n,\A_n$. The Dynkin diagrams of the semisimple Lie algebras of
rank $n$ are denote by $A_n,B_n,C_n,D_n,E_n,F_4,G_2$ . The multiplicative group
of the field $\k$ is denoted by $\k^\times$,
while the dual group is denoted $\Gamma^*=\Hom(\Gamma,\k^\times)$. We frequently
call the generator of the multiplicatively denoted group
$\Z_2=\langle\theta\rangle$.\\

The following notions are standard. We summarize them to fix notation and
refer to \cite{HLecture08} for a detailed account.

\begin{definition}
  A \emph{\ydmP}\index{Yetter-Drinfel'd module $M$} $M$ over a group $\Gamma$
  is a $\Gamma$-graded vector space,
  $M=\bigoplus_{g\in \Gamma} M_g$
  with a $\Gamma$-action on $M$ such that
  $g.M_h=M_{ghg^{-1}}$
  Note that over $\k=\C$ any \ydm $M$ is semisimple, i.e. the direct sum of $n$
  simple \ydmsP, and we call $n$ \emph{rank}. Call $M$ 
  \begin{itemize}
    \item \emph{(link-) indecomposable}\index{Link
    indecomposable (YDM)}, iff the support
    $\{g\;|\;M_g\neq 0\}$ generates all $\Gamma$.
    \item \emph{minimally indecomposable}\index{Minimally
    (link-) indecomposable (YDM)}, iff $M$ is indecomposable and no proper
    sub-\ydm is indecomposable. Every indecomposable \ydm contains a
    minimally indecomposable one.
    \item \emph{faithful}\index{Faithful
  (YDM)}, iff the $\Gamma$-action on $M$ is faithful.
  \end{itemize}
\end{definition}

\begin{lemma}
  The map $c_{MM}:\;M\otimes M\rightarrow M\otimes M$ defined by
  $$M_g\otimes M_h\ni\; v\otimes w\stackrel{c_{MM}}{\longmapsto}
    g.w\otimes v\;\in M_{ghg^{-1}}\otimes M_{g}$$
  fulfills the \emph{Yang-Baxter-equation}
  $$(id\otimes c_{MM})(c_{MM}\otimes id)(id\otimes c_{MM})
    =(c_{MM}\otimes id)(id\otimes c_{MM})(c_{MM}\otimes id)$$
  turning $M$ into a \emph{braided vector space} with \emph{braiding} $c_{MM}$.
\end{lemma}

\begin{example}
For \emph{abelian} groups $\Gamma$, the compatibility condition implies the
stability
of the homogeneous components $M_g$. For $\k=\C$ all {simple} \ydms $M_i$
are 1-dimensional and isomorphic to some $\O_{g_i}^{\chi_i}:=x_i\k$ with
$\Gamma$-graduation $g_i$ and $\Gamma$-action defined by a 1-dimensional
character $\chi_i:\;\Gamma\rightarrow \k^\times$ via $g.x_i:=\chi_i(g)x_i$. The
braiding $c_{MM}$ is hence \emph{diagonal} with \emph{braiding matrix}
$q_{ij}:=\chi_j(g_i)$.
$$x_i\otimes x_j\stackrel{c_{MM}}{\longmapsto} q_{ij}(x_j\otimes
x_i)$$
\end{example}

\begin{definition}
  Let $M$ be a \ydm over an arbitrary group $\Gamma$ and let $e_k\in M_{g_k}$
  be a fixed homogeneous basis; for $\Gamma$ abelian one may choose $e_k:=x_k$.
  Consider the tensor algebra $T(M)$, which
  can be identified with the algebra of words in the letters
  $\{e_k\}_{k=1\ldots \dim(M)}$ and is
  again a $\Gamma$-\ydmP. We can uniquely obtain
  \emph{skew derivations} $\partial_i:\;T(M)\rightarrow T(M)$ by
  $$\partial_k(1)=0
  \qquad\partial_k(e_l)=\delta_{kl}1
  \qquad\partial_k(x\cdot y)=\partial_k(x)\cdot (g_k.y)+x\cdot \partial_k(y)$$
  The \emph{Nichols algebra}\index{Nichols algebra $\B(M)$} $\B(M)$ is the
  quotient of $T(M)$ by the largest
  homogeneous ideal $\mathfrak{I}$ in degree $\geq 2$, invariant under
  all $\partial_k$. It is a $\Gamma$-\ydmP.  
\end{definition}

Following \cite{HLecture08} we draw a {$q$-diagram} for \ydm $M$ over an
abelian group by
drawing a node for each basis element $x_i$ spanning a corresponding
1-dimensional simple summand $M_i=\O_{g_i}^{\chi_i}=x_i\k$ of $M$. We draw an
edge between $x_i,x_j$ whenever $q_{ij}q_{ji}\neq 1$ (i.e. $c_{M_iM_j}^2\neq
id$) and decorate each node $i$ by the complex numbers $q_{ii}$ and each
edge $ij$ by $q_{ij}q_{ji}$. It turns out that this data is all that is needed
to determine the Nichols algebra $\B(M)=\B(\bigoplus_{i\in I}M_i)$.

\begin{definition}
  The \emph{adjoint action} $\B(M)\otimes\B(M)\to\B(M)$ is given by
  $$x\otimes y\longmapsto\ad{x}{y}:=x^{(1)}yS(x^{(2)})$$
  For \ydms $N,L\subset \B(M)$ we define the \emph{ad-space} by 
  $$\ad{N}{L}:=\{\ad{x}{y}\;|\;x\in N,\;y\in L\}$$
\end{definition}
The root system theory in \cite{AHS10} describes $\B(M)$ in terms
of Nichols algebras over iterated ad-spaces between the simple summands $M_i$
(simple roots). Especially if $M$ is over an abelian group, all $M_i=x_i\k$ are
$1$-dimensional and we obtain a (slightly different) PBW-basis of iterated
braided commutators, as already observed by Kharchenko \cite{Kha08}. Compare the
lecture notes \cite{HLecture08}.
\begin{definition}\label{def_Cartan}
  For a finite-dimensional Nichols algebra $\B(M)$ with semisimple \ydm
  $M=\bigoplus_{i\in I} M_i$ consider the following structure constants for
  $i\neq j\in I$:
  $$C_{i,j}:=\max_m\left(\adn{m}{M_i}{M_j}:=\ad{M_i}{\ad{M_i}{
  \cdots M_j}}\neq \{0\}\right)$$
  Together with $C_{i,i}:=2$ they define a \emph{Cartan matrix}
  \footnote{Be warned, that this corresponds to choosing for Lie algebras
  $C_{i,j}=\frac{2(\alpha_j,\alpha_i)}{(\alpha_i,\alpha_i)}$ according to Kac,
  Jantzen, Carter etc. and is transpose to the convention e.g. in
  Humphreys book.}
  $(C_{i,j})_{i,j\in I}$ of $M$.
  One may draw a \emph{Dynkin diagram} with node set $I$ and edges decorated
  by $(C_{i,j},C_{j,i})$.\\
  For some cases pictorial representations
  are custom, e.g. double line and arrow for $(-1,-2)$ as in the root system
  $B_2$.
\end{definition}
\begin{theorem}\label{thm_Cartan}
 For a finite-dimensional Nichols algebra $\B(M)$ over an abelian group with
  diagonal braiding matrix $q_{ij}$ the Cartan matrix can be equivalently
  obtained by 
  $C_{i,j}:=\min_m\left(q_{ii}^{-m}=q_{ij}q_{ji}\;or\;q_{ii}^{m+1}=1 \right)$,
  see e.g. \cite{HLecture08} Prop. 5.5.
\end{theorem}

If especially the Cartan matrix $C_{i,j}$ is the
Cartan matrix of a semisimple Lie algebra $\mathfrak{g}$, then the Nichols
algebra has a PBW-basis consisting of monomials in iterated braided commutators
according to the positive roots in the classical root system of
$\mathfrak{g}$. However, several additional exotic
examples of finite-dimensional Nichols algebras exist, that possess unfamiliar
Dynkin diagrams, such as a multiply-laced triangle, and where Weyl reflections
may connect different $\Gamma$-\ydms (yielding a {Weyl groupoid}). Heckenberger
completely classified all Nichols algebras over abelian $\Gamma$ in
\cite{Heck09}.

\section{Covering Nichols Algebras}

From now on, we always suppose a central extension of an abelian
group $\Gamma$:
$$1\rightarrow \Sigma^* \rightarrow G\stackrel{\pi}{\longrightarrow}
  \Gamma\rightarrow 1 \qquad\Sigma\subset Z(\Gamma)$$
where $\Sigma^*=\Hom(\Sigma,\k^\times)$. We denote elements in $G$ by letters
such as $g$, whereas elements in the abelian quotient $\Gamma$ are denoted by
$\bar{g}$. A $\Gamma$-\ydm is denoted by $M$, whereas the covering $G$-\ydm
will be denoted by $\tilde{M}$.

\subsection{Twisted Symmetries}\label{sec_twistedSymmetry}

\begin{definition}
  Let $M$ be a \ydm over an abelian group $\Gamma$
  $$M=\bigoplus_{i\in I} M_i=\bigoplus_{i\in I}
  \O_{\bar{g}_i}^{\chi_i}=\bigoplus_{i\in I} x_i\k
    \qquad q_{ij}=\chi_j(\bar{g}_i)$$ 
  written as a sum of simple 1-dimensional \ydms $M_i,\;i\in I$.\newline 
  We call a linear bijection 
  $f_0:\;M\rightarrow  M$ a \emph{twisted symmetry},  iff 
  \begin{itemize}
    \item $f_0$ preserves the $\Gamma$-grading (resp. is colinear)
    \item $f_0$ preserves self-braiding and monodromy:
      $$(f_0\otimes f_0)c_{M_iM_i}=c_{M_iM_i}(f_0\otimes f_0)$$
      $$(f_0\otimes f_0)c_{M_iM_j}c_{M_jM_i}=
	c_{M_iM_j}c_{M_jM_i}(f_0\otimes f_0)$$
  \end{itemize}
\end{definition}
The braiding itself needs not to be preserved. However, we wish to
control the modification by the following notions, regardless of $\Gamma$ being
abelian:
\begin{definition} \label{def_TwistedSymmetryWithRespect}
Let $M=\bigoplus_{i\in I} M_i$ be a semisimple,
indecomposable $\Gamma$-\ydm decomposed into 1-dimensional summands $M_i$.
\begin{itemize}
 \item We call $f_0$ a twisted symmetry with respect to a
  given bimultiplicative form
  $\langle,\rangle:\;\Gamma\times\Gamma\rightarrow\k^\times$,
  iff $f_0$ is a twisted symmetry and 
  $$c_{M_iM_j}(f_0\otimes f_0)=
    \langle \bar{g}_i,\bar{g}_j\rangle(f_0\otimes f_0)c_{M_iM_j}$$
  Note that since $f_0$ preserves the monodromy $c_{M_iM_j}c_{M_jM_i}$, the
  form is always skew-symmetric $\langle \bar{g},\bar{h}\rangle=\langle
  \bar{h},\bar{g}\rangle^{-1}$ and since  $f_0$ preserves the self-braiding
  $c_{M_iM_i}$, the form is isotropic  $\langle \bar{g},\bar{g}\rangle=1$.
  It is hence \emph{symplectic}.
 \item We call $f_0$ a twisted symmetry with respect
  to a group-2-cocycle $\sigma_0\in Z^2(\Gamma,\k^\times)$, iff it is a twisted
  symmetry with respect to the form
  $$\langle \bar{g},\bar{h}\rangle 
  :=\sigma_0(\bar{g},\bar{h})\sigma_0^{-1}(\bar{h},\bar{g})$$
\end{itemize}
\end{definition}

\begin{corollary}\label{cor_twistedAction}
  Any twisted symmetry $f_0:M\rightarrow M$ of an indecomposable \ydm $M$ with
  respect to a given form  $\langle,\rangle:\;\Gamma\times\Gamma\rightarrow \k$
  is an isomorphism of  \ydms $M\rightarrow M_{\langle,\rangle}$, where
  $M_{\langle,\rangle}$ is $M$  as $\Gamma$-graduated vector space and the
  $\Gamma$-action on a homogeneous element $v\in M_{\bar{h}}$ is modified to
  $$\bar{g}._{\langle,\rangle}v:=\langle \bar{g},\bar{h}\rangle(\bar{g}.v)$$
\end{corollary}

\begin{remark}\label{rem_twistedAction}
  In particular, any twisted symmetry $f_0$ of an indecomposable \ydm $M$ with
  respect to a given group-2-cocycle $\sigma_0\in Z^2(\Gamma,\k^\times)$ is an
  isomorphism of \ydms $M\rightarrow M_{\sigma_0}$, where $M_{\sigma_0}$ has
  accordingly modified $\Gamma$-action
  $$\bar{g}._{\sigma_0}v:=
    \sigma_0(\bar{g},\bar{h})\sigma_0^{-1}(\bar{h},\bar{g})(\bar{g}.v)$$
  or equivalently twisted characters $\chi_i^{\sigma_0}(\bar{g}):=
    \sigma_0(\bar{g},\bar{h})\sigma_0^{-1}(\bar{h},\bar{g})
    \chi^{\sigma_0}_i(\bar{g})$.
  By e.g. \cite{Mas08} Prop 5.2 this condition precisely means that
  $f_0$ can  be extended to an isomorphism of Nichols algebras $\B(M)\rightarrow
  \B(M_{\sigma_0})$ to the Doi twist $\B(M_{\sigma_0})\cong \B(M)_{\sigma_0}$.
\end{remark}
Next, we consider a family of twisted symmetries that respect a group
law:

\begin{definition}\label{def_TwistedSymmetryGroup}
  Let $M$ be a \ydm over an abelian group $\Gamma$ and $\Sigma$ a
  finite abelian group. We say $\Sigma$ \emph{acts as twisted
  symmetries} on $M$ iff 
  \begin{itemize}
    \item $\Sigma$ acts on the vector space $M$, i.e. denoting the action of an
      $p\in\Sigma$ by $f_p:\;M\rightarrow M$ we demand $f_1=id_M$
      and $f_pf_q=f_{pq}$ for all $p,q\in\Sigma$.
    \item For each $p\in\Sigma$ the action $f_p$ is a twisted symmetry.
  \end{itemize}
  For a given group-2-cocycle $\sigma\in Z^2(\Gamma,\Sigma^*)$, we say that
  $\Sigma$ acts on $M$ as twisted symmetry with respect to $\sigma$, iff
  for all $p\in\Sigma$ the twisted symmetry $f_p$ of $M$ is a twisted symmetry
  with respect to the group-2-cocycle 
  $$\sigma_p(\bar{g},\bar{h}):=\sigma(\bar{g},\bar{h})(p)$$
\end{definition}

\subsection{Main Construction Theorem}\label{sec_Construction}

Suppose a given central extension of an abelian
group $\Gamma$:
$$1\rightarrow \Sigma^* \rightarrow G\stackrel{\pi}{\longrightarrow}
  \Gamma\rightarrow 1 \qquad\Sigma\subset Z(\Gamma)$$
It can be described in terms of a cohomology class of 2-cocycles 
$[\sigma]\in H^2(\Gamma,\Sigma^*)$.
We fix a set-theoretic section
$s:\Gamma\rightarrow G$ of $\pi$ which is normalized, i.e. $s(1)=1$. This
corresponds to a choice of a specific representing 2-cocycle $\sigma\in
Z^2(\Gamma,\Sigma^*)$ with
$s(\bar{g})s(\bar{h})=\sigma(\bar{g},\bar{h})s(\bar{g}\bar{h})$. Different
choices of $s,\sigma$ will in what follows produce identical forms $\langle
\bar{g},\bar{h}\rangle_p 
  :=\sigma(\bar{g},\bar{h})(p)\sigma^{-1}(\bar{h},\bar{g})(p)$ and hence
identical notions of twisted symmetry.

\begin{theorem}[Covering Construction]\label{thm_Construction}
  Suppose now $M$ to be a \ydm over $\Gamma$ and an action of $\Sigma$
  on $M$ as twisted symmetries with respect to the
  $\sigma\in Z^2(\Gamma,\Sigma^*)$
  fixed above. Because $\Sigma$ is abelian, we may
  simultaneously diagonalize the action and decompose $M$ into eigenspaces
  $M^{[\lambda]}$ with simultaneously eigenvalues $\lambda\in\Sigma^*$. Then
  the following structures define a $G$-\ydm $\tilde{M}$, which 
  we call the \emph{covering \ydmP} of $M$:
  \begin{itemize}
    \item $\tilde{M}:=M$ as vector space
    \item The $G$-action is the pullback of the $\Gamma$-action via
      $\pi$.\newline
      Especially $\Sigma^*\subset G$ acts trivially and thus $\tilde{M}$ is
      not faithful.
    \item The eigenspaces $M^{[\lambda]}$ give rise to the $G$-homogeneous
     layers via
      $\tilde{M}_h:=M^{[hs(\bar{h})^{-1}]}_{\bar{h}}$.
      Note that $\pi(hs(\bar{h})^{-1})=1$ so, $\lambda:=hs(\bar{h})^{-1}$ is
      indeed an element of $\Ker(\pi)=\Sigma^*$.
  \end{itemize}
\end{theorem}
\begin{proof}
  To prove $M$ to be a well-defined \ydmP, we have to check
  that $\tilde{M}$ fulfills the (nonabelian) Yetter-Drinfel'd condition
  $g.\tilde{M}_{h}=\tilde{M}_{ghg^{-1}}$:\\

  {\bf Claim 1:} The $\Gamma$-action permutes simultaneous eigenspaces
  $M^{[\lambda]}$ as follows:
  $$\bar{g}.M_{\bar{h}}^{[\lambda]}=M_{\bar{h}}^{[   
  \sigma(\bar{g},\bar{h})\sigma^{-1}(\bar{h},\bar{g})\cdot \lambda]}$$
  This can just be calculated: Let $v\in M_{\bar{h}}^{[\lambda]}$, i.e. the
  twisted symmetry action of any $p\in\Sigma$ is
  $f_p(v)=\lambda(p)v$, then by the defining property of
  a twisted symmetry 
  \begin{align*}
    f_p(\bar{g}.v)
    &=\sigma_p(\bar{g},\bar{h})\sigma_p^{-1}(\bar{h},\bar{g})
      \bar{g}.f_p(v)\\
    &=\sigma_p(\bar{g},\bar{h})\sigma_p^{-1}(\bar{h},\bar{g})\cdot
      \lambda(p)v\\
    &=\left(\sigma(\bar{g},\bar{h})\sigma^{-1}(\bar{h},\bar{g})
      \cdot\lambda\right)(p)\cdot v
  \end{align*}
  and thus $\bar{g}.v$ is a simultaneous eigenvector of the $\Sigma$-action with
  eigenvalues $\sigma_p(\bar{g},\bar{h}) 
  \sigma_p^{-1}(\bar{h},\bar{g})\cdot \lambda$
  as claimed.\\
 
  {\bf Claim 2:} $\Gamma$ abelian implies the commutator
  in $G$ can be expressed as
  $$[G,G]\ni[g,h]=\sigma(\bar{g},\bar{h}) 
  \sigma^{-1}(\bar{h},\bar{g})\in \Sigma^*$$
  This is by definition of $\sigma$ true for elements $s(\bar{g}),s(\bar{h})$ in
  $\Im(s)$: 
  \begin{align*}
    [s(\bar{g}),s(\bar{h})]
    &=s(\bar{g})s(\bar{h})\cdot
      s(\bar{g})^{-1}\cdot s(\bar{h})^{-1}\\
    &=s(\bar{g})s(\bar{h})\cdot 
      s(\bar{g}\bar{h})^{-1}s(\bar{g}\bar{h})^{-1} \cdot
      s(\bar{g})^{-1}\cdot s(\bar{h})^{-1}\\
    \text{($\Gamma$ abelian)$\quad$}
    &=s(\bar{g})s(\bar{h})s(\bar{g}\bar{h})^{-1}\cdot
      s(\bar{h}\bar{g})^{-1}s(\bar{g})^{-1}\cdot s(\bar{h})^{-1}\\
    &=\sigma(\bar{g},\bar{h})\cdot\sigma^{-1}(\bar{h},\bar{g})
  \end{align*}
  General elements $g,h\in G$ differ from such elements in $\Im(s)$ by a factor
  in $\Ker(\pi)=\Sigma^*\subset G$. Because $\Sigma^*$ was supposed central in
  $G$, this does not change the commutator $[g,h]$, while the right hand side of
  the claim anyway only depends on the images $\bar{g},\bar{h}\in \Gamma$. Thus
  the claim holds in for general $g,h\in G$ as well.

  Claims 1 and 2 imply the asserted Yetter-Drinfel'd condition
  $g.\tilde{M}_h=\tilde{M}_{ghg^{-1}}$
  \begin{align*}	
      g.\tilde{M}_h
      &=\bar{g}.M^{[hs(\bar{h})^{-1}]}_{\bar{h}}\\
      \text{(claim 1)$\quad$}
      &=M^{[\sigma(g,h)\cdot\sigma^{-1}(h,g)
	\cdot hs(\bar{h})^{-1}]}_{\bar{h}}\\
      \text{(claim 2)$\quad$}
      &=M^{[ghg^{-1}h^{-1}
	\cdot hs(\bar{h})^{-1}]}_{\bar{h}}\\ 
      \text{($\Gamma$ abelian)$\quad$}
      &=M^{[ghg^{-1}s(\bar{g}\bar{h}\bar{g}^{-1})^{-1}]}_{\bar{h}}
      = \tilde{M}_{ghg^{-1}}
  \end{align*}
\end{proof}

If a central extension $\Sigma^*\rightarrow G\rightarrow\Gamma$ is a stem
extension $\Sigma^*\subset [G,G]$, then it is an easy group theoretic fact
that any preimage of any generating system of $\Gamma$ generates $G$. Hence:

\begin{corollary}
  For a stem extension, the covering $\tilde{M}$ of an indecomposable
  $\Gamma$-\ydm $M$ is an indecomposable $G$-\ydmP.
\end{corollary}

By construction the $G$-action on $\tilde{M}$ factorizes to the
$\Gamma$-action on $M$, thus:

\begin{corollary}\label{cor_braidingIsomorphic}
  $M,\tilde{M}$ are isomorphic as braided vector spaces. Especially
  $\B(M),\B(\tilde{M})$ are isomorphic as $\Z$-graded  algebras (see e.g.
  \cite{AS02} Section 5.1).
\end{corollary}

\begin{remark}
  Note that in \cite{Len12} we gave a much more general construction:
  \begin{itemize}
    \item A Hopf algebra $H$ and a group $\Sigma$ of Bigalois objects
      $H_p$ yield a Hopf algebra structure on the direct sum, the covering
      Hopf algebra (see \cite{Len12} Thm. 1.6): 
      $$\Omega:=\bigoplus_{p\in\Sigma} H_p$$
      fitting into an exact sequence of Hopf algebras (see \cite{Len12} Thm.
      1.13)
      $$0\rightarrow \k^\Sigma\rightarrow \Omega \rightarrow H\rightarrow 0 $$
      The covering Hopf algebra $\Omega$ thereby can only be pointed, if among
      others $H$ is pointed and $\Sigma$ is an abelian group.  
    \item If we specialize this to the bosonization $H=\k[\Gamma]\#\B(M)$
      of a Nichols algebra of a \ydm $M$ over an arbitrary group $\Gamma$,
      we yield a covering Nichols algebra $\B(\tilde{M})$ over a central     
      extension. 
      $$1\rightarrow \Sigma^* \rightarrow G \rightarrow
      \Gamma\rightarrow 1$$
      The construction in \cite{Len12} Thm. 4.3 uses a newly defined coaction.
      The direct formulation in the Construction Theorem
      \nref{thm_Construction} follows after diagonalizing the twisted
      symmetries. As an example for
      $\Gamma$ nonabelian, we have also constructed e.g. a Nichols algebra of
      dimension $24^2$ over $GL_2(\F_3)\rightarrow\S_4$.
  \end{itemize}
\end{remark}

\subsection{Impact On The Dynkin Diagram: Folding}

First we observe, that if a twisted symmetry of a $\Gamma$-\ydm $M$ directly
permutes the simple summands $M_i$ (and thus the index set $I$), then it is
already an automorphism of the $q$-diagram of $M$:

\begin{definition}
  Let $M=\bigoplus_{i\in I} M_i=\bigoplus_{i\in I}x_i$ be a vector space
  decomposed  into $1$-dimensional sub-vector spaces according to some index set
  $I$, e.g.  the decomposition of a \ydm over an abelian group $\Gamma$
  decomposed  into simple summands. Then an action
  $\left(f_p\right)_{p\in\Sigma}$ of a group $\Sigma$ on  $M$ is called
  \emph{permutation action}, iff it is induced by a permutation representation 
  $\rho$ on $I$, i.e.
  $$\rho:\;\Sigma\rightarrow \Aut(I)=\S_{|I|}
    \qquad \forall_{i\in I}\;f_p(x_i)=x_{p.i}$$
  If $M$ is moreover a diagonally braided vector space with braiding matrix 
  $q_{ij}$ with respect to the basis $\{x_i\}_{i\in I}$, then we denote the
  permuted braiding matrix by $(q^{(p)})_{ij}:=q_{p.i,p.j}$.
\end{definition}

\begin{remark}
  Note without proof that the technical assumption in the previous
  definition is an implicit necessity to obtain minimally
  indecomposable Nichols algebras. Certain non-minimally indecomposable Nichols
  algebras may however require the consideration of non-permutation actions
  (e.g in Remark \nref{rem_class3}).
\end{remark}

We may now for abelian $\Gamma$ express the twisted symmetry condition of a
given permutation action in terms of the permuted braiding matrix. We
especially recognize $\Sigma$ to consist
necessarily of automorphism of the $q$-diagram:

\begin{corollary}\label{cor_diagramAutomosphisms}
  Let  $M=\bigoplus_{i\in I} M_i=\bigoplus_{i\in I}x_i$ be a \ydm over an
  abelian group $\Gamma$ and $\Sigma$ a group with permutation action
  on the decomposed $M$. Then the action is a action by twisted symmetries
  according to Definition \nref{def_TwistedSymmetryGroup} iff
    \begin{itemize}
      \item The $\Sigma$-action on $I$ only permutes elements $i,j$ with
	$M_i,M_j$ in the same homogeneous component of $M$.
      \item The $\Sigma$-action on $I$ preserves the self-braiding
	$$(q^{(p)})_{ii}=q_{ii}$$
      \item The $\Sigma$-action on $I$ preserves the monodromy
	$$(q^{(p)})_{ij}(q^{(p)})_{ji}=q_{ij}q_{ji}$$
    \end{itemize}
    Moreover, the $\Sigma$-action is a twisted symmetry with respect to
	a 2-cocycle $\sigma\in Z^2(\Gamma,\Sigma^*)$, iff the braiding matrix
	transforms under the action of all $p\in\Sigma$ according to the
	prescribed bimultiplicative form  
	$$\langle \bar{g},\bar{h}\rangle_p 
	  :=\sigma(\bar{g},\bar{h})(p)\sigma^{-1}(\bar{h},\bar{g})(p)$$
	$$q_{ij}^{(p)}=\langle \bar{g}_i,\bar{g}_j \rangle_p
	  q_{ij}$$
    Hence especially the permutation action is an automorphism of the
    $q$-diagram of $M$,
    that has by definition node set $I$ and is decorated with $q_{ii}$ and
    $q_{ij}q_{ji}$.
\end{corollary}

% 
% We furthermore prove, that the condition in the preceding lemma can be
% assumed, if we are looking for minimally indecomposable covering \ydms over
%$G$.
% Note that such a minimally indecomposable $G$-\ydm is contained in every
% indecomposable $G$-\ydm.
% 
% \begin{lemma}
%   Suppose $\Sigma^*\rightarrow G\rightarrow \Gamma$ a stem-extension with
% group-2-cocycle $\sigma\in Z^2(\Gamma,\Sigma^*)$ and $M$ to be a $\Gamma$-\ydm
% with an action of $\Sigma$ by twisted symmetries with respect to $\sigma$.
% Assumed that the covering \ydm $\tilde{M}$ over $G$ is minimally
%indecomposable,
% then all 1-dimensional simple summands $M_i$ of $M$ are mutually
%non-isomorphic
% and hence by the preceding lemma the $\Sigma$-action corresponds to $\Sigma$
% acting on the $q$-diagram of $M$ by graph automorphisms.
% \end{lemma}
% \begin{proof}
%   We assumed $\tilde{M}$ to be minimally indecomposable. Because
%$[G,G]\subset
% Z(G)$ As by definition  $\tilde{M}$
%   Assume conversely $M_i\cong M_j$ as \ydm for some $i\neq j\in I$. Then
%   especially the $\Gamma$-graduation coincides $\bar{g}_i=\bar{g}_j$. Consider
%   the generated $\Sigma$-submodules $\Sigma.M_i$ resp. $\Sigma.M_j$, that
%   are obviously $\Gamma$-sub-\ydmsP. \\
% 
%   {\bf Claim:} $\Sigma.M_i \cap \Sigma.M_j=\{\}$. Assume conversely that some
%   1-dimensional simple \ydm $M_k=x\k$ with $k\in I$ is contained in
%$\Sigma.M_i
%   \cap \Sigma.M_j$. 
%   Suppose   
% Because  because   that cannot coincide,
% because  are linearly independent $\Sigma$- hence 
% 
% \end{proof}

Next we calculate the Dynkin diagram of the covering Nichols algebra
$\B(\tilde{M})$ of a Nichols algebra $\B(M)$ of a \ydm $M$ over an abelian group
$\Gamma$. We restrict to specific scenarios appearing in the present
article (especially $p=2$), but similar calculations can be carried out for
other situations as well.

\begin{definition}
  Let $M=\bigoplus_{i\in I} M_i=\bigoplus_{i\in I}\O^{\chi_i}_{\bar{g}_i}$ be
  some $\Gamma$-\ydm and let $\Sigma=\Z_2=\langle \theta\rangle$  act
  on $M$ as twisted permutation symmetries. Then $I$ decomposes into orbits of
  length $1$ resp. $2$. We call such nodes $i\in I$
  \emph{inert} resp. \emph{split}.
\end{definition}
In this situation the splitting behaviour of a node is determined by whether
it's decoration is central in $G$:

\begin{tabular}{p{10cm}p{3cm}p{3cm}}
 \begin{lemma}\label{lm_SplitInert}
  Let $\Sigma$ acts as twisted permutation symmetries: If a node
  $i\in I$ is inert, then necessarily its decoration $\bar{g}_i\in\Gamma$
  has central image $s(\bar{g}_i)\in G$. If  moreover all
  1-dimensional simple summands $M_i$ of $M$ are mutually non-isomorphic, then
  the converse holds also: A node $i\in
  I$ is inert resp. split iff $s(\bar{g}_i)\in G$ is central resp. noncentral in
  $G$.\newline

  Hence $\tilde{M}$ decomposes into simple $G$-\ydms
  $\tilde{M}_{\tilde{k}}$, where the new nodes $\tilde{k}\subset I$ are
  $\Sigma$-orbits of cardinality $1$ resp. $2$ and are called inert resp. split
  as well. These subsets of $I$ hence form the new nodes set $\tilde{I}$.
  \end{lemma}
  &\hspace{0.3cm}\raisebox{-\totalheight}{\imageplain{0.26}{inert}}
  &  \hspace{-.5cm}\raisebox{-\totalheight}{\imageplain{0.26}{split}}
\end{tabular}

\begin{proof}
  Because we consider an action of $\Sigma=\Z_2$ on a set $I$, the orbits have
  length $1,2$.\\

  For the {\bf first claim}, we assume some $i_1\in I$ with
  $s(\bar{g}_{i_1})\not\in
  Z(G)$ and prove $i_1$ to be a split node: By assumption of noncentrality
  there exist some $g\in G$ with commutator
  $$[g,s(\bar{g}_{i_1})]=\theta^*\in \Sigma^*\subset G$$
  This is, because $\Gamma$ is abelian and hence every commutator lays in the
  kernel $\Sigma^*$ of the central extension; if the commutator is nontrivial,
  it has to coincide with the generator $\theta^*$ of $\Sigma^*\cong \Z_2$,
  i.e. the element with $\theta^*(\theta)=-1_\k$. By
  claim 2 in the proof of Theorem \nref{thm_Construction} we then have 
  $$\sigma(\bar{g},\bar{g}_{i_1})\sigma^{-1}(\bar{g}_{i_1},\bar{g})
    =[g,s(\bar{g}_{i_1})]=\theta^*$$
  We assumed $\Sigma$ to act by twisted permutation symmetries $f_p$, hence
  $f_\theta(M_{i_1})=:M_{i_2}$ is another summand of $M$ and we wish to prove
  $i_1\neq i_2$. This finally follows from Remark \nref{rem_twistedAction},
  as the twisted $\Gamma$-action on $M_{i_2}$ is 
  $$\bar{g}._{\sigma_\theta}v:=\sigma_\theta(\bar{g},\bar{h})
    \sigma_\theta^{-1}(\bar{h},\bar{g})\cdot\bar{g}.v
    =\sigma(\bar{g},\bar{h})
    \sigma^{-1}(\bar{h},\bar{g})(\theta)\cdot\bar{g}.v
    =\theta^*(\theta)\cdot\bar{g}.v=-\bar{g}.v$$
  This is a different $\Gamma$-action, hence $M_{i_1}\not\cong
  M_{i_2}$ and $i_1\neq i_2$ and thus the node is split.\\

  For the {\bf second claim}, assume now moreover that all 1-dimensional simple
  summands $M_i$ of $M$ are
  mutually non-isomorphic, then we also prove the converse: Suppose some
  $i_1\in I$ with $s(\bar{g}_{i_1})\in Z(G)$, then we prove $i_1$ to be inert:
  By assumption of centrality the commutator $[g,s(\bar{g}_i)]=1$ for all $g\in
  G$. By  claim 2 in the proof of Theorem \nref{thm_Construction} we have 
  $$\sigma(\bar{g},\bar{g}_{i_1})\sigma^{-1}(\bar{g}_{i_1},\bar{g})
    =[g,s(\bar{g}_{i_1})]=1_{\Sigma^*}$$
  We assumed $\Sigma$ acts as twisted permutation symmetries,
  hence $f_\theta(M_{i_1})$ is also a summand $M_{i_2}$ of $M$; we wish to
  prove $i_1= i_2$. By Remark \nref{rem_twistedAction}, the twisted
  $\Gamma$-action on $M_{i_2}$ is:
  $$\bar{g}._{\sigma_p}v:=\sigma_p(\bar{g},\bar{h})
    \sigma_p^{-1}(\bar{h},\bar{g})(\bar{g}.v)=\bar{g}.v$$
  Thus $M_{i_1}\cong M_{i_2}$ and by the additional assumption hence
  $i_1=i_2$ and $i_1$ is inert.
\end{proof}

\newpage

\begin{theorem}\label{thm_Fold}
  Let $M=\bigoplus_i M_i$ be a \ydm over the abelian group $\Gamma$ and
  $\Z_2\to  G\to\Gamma$ a stem extension as above. Suppose again the simple
  summands $M_i$ to be mutually  nonisomorphic and consider the covering 
  $\tilde{M}=\bigoplus_{\tilde{i}\in\tilde{I}}\tilde{M}_{\tilde{i}}$
  constructed in Theorem \nref{thm_Construction}. The set $\tilde{I}$ of
  simple summands $\tilde{M}_{\tilde{i}}$ has been described in the previous
  Lemma \nref{lm_SplitInert}.\\

  We now assume several situations (relevant to this article) for the
  Cartan matrix resp. Dynkin diagram of $M$ and calculate in these
  situations Cartan matrix $\tilde{C}_{\tilde{k},\tilde{l}}$ and
  hence Dynkin diagram of the covering $\tilde{M}$ (``folded diagrams''):\\
  \begin{enumerate}
    \item {\bf Disconnected:} 
      Let $\tilde{k},\tilde{l}\in\tilde{I}$ be arbitrary nodes
      (split or inert) and suppose all elements $k\in\tilde{k}\subset I$ are
      disconnected to all elements in $l\in\tilde{l}\subset I$ i.e.
      $q_{kl}q_{lk}=1$ and $\ad{x_k}{x_l}=0$. Equivalently we may assume the 
      q-subdiagram of $M$ to consist of mutually disconnected components
      $\tilde{k}\times\tilde{l}$. Then the Cartan matrix of
      $\B(\tilde{M})$ is also diagonal $\tilde{C}_{\tilde{k}\tilde{l}}=0$,
      i.e. the covering nodes $\tilde{k},\tilde{l}$ of $\tilde{M}$ are	
      disconnected as well.\\

    \item {\bf Inert Edge:} Let $\tilde{k}=\{k\},\tilde{l}=\{l\}\subset I$
      be inert nodes. Then the Cartan matrix entry in the covering Nichols
      algebra is of identical type $\tilde{C}_{\tilde{k},\tilde{l}}=C_{k,l}$.\\

    \item {\bf Split Edge:} Let $\tilde{k}=\{k_1,k_2\}$ and
      $\tilde{k}=\{k_1,k_2\}$
      both be split and (after possible renumbering) let
      the Dynkin diagram of $\B(M)$ restricted to the $4$ simple
      $\Gamma$-\ydms $M_{k_1},M_{k_2},M_{l_1},M_{l_2}$ be of type $A_2\times
      A_2$. This means by definition: \newline
      \begin{tabular}{p{10cm}p{3cm}}
      \begin{center}$\ad{M_{k_i}}{M_{l_i}}=:N_i\neq\{0\} \qquad
	i=1,2$\end{center}$\quad$\newline
      where $N_i\subset \B(M)$ and all other ad-spaces are trivial.
%       $A^{q=-1}_2\times A^{q=-1}_2$, i.e.
% 	  $$q_{k_1k_1}=q_{k_2k_2}=q_{l_1l_1}=q_{l_2l_2}=-1$$
% 	  $$q_{k_1l_1}q_{l_1k_1}=-1\qquad q_{k_2l_2}q_{l_2k_2}=-1$$
% 	  $$q_{k_1l_2}q_{l_2k_1}=+1\qquad q_{k_2l_1}q_{l_1k_2}=+1$$
      Then the Cartan matrix entry in the covering Nichols algebra over $G$ is
      of type $A_2$ with $\ad{\tilde{M}_k}{\tilde{M}_l}=N_1 \oplus N_2
      \subset\B(\tilde{M})$\newline
      \begin{center}$\left(\begin{matrix}
	  C_{\tilde{k}\tilde{k}} & C_{\tilde{k}\tilde{l}}\\
	  C_{\tilde{l}\tilde{k}} & C_{\tilde{l}\tilde{l}}
	  \end{matrix}\right)
	  =
	  \left(\begin{matrix}
	  2 & -1\\
	  -1 & 2
	\end{matrix}\right)$\end{center}$\quad$\newline
	See the example over $G\cong \D_4\rightarrow \Z_2^2\cong\Gamma$ in
	Section \nref{sec_D4}.
&\hspace{0.5cm}\raisebox{+0.00\totalheight}{\imageplain{0.26}{DynkinA2A2A2}} 
      \end{tabular}
\newpage 
    \item {\bf Ramified Edge:} Let $\tilde{k}=\{k_1,k_2\}\in\tilde{I}$ be split
      and $\tilde{l}=\{l\}$ be inert. Let
      the Dynkin diagram of $\B(M)$ restricted to the $3$ simple $\Gamma$-\ydms
      $M_{k_1},M_{k_2},M_{l}$ be of type $A_3$
      with $l$ the middle node, i.e.\newline
	\begin{center}
	$\ad{M_{k_i}}{M_{l}}=:N_i\neq\{0\} \qquad i=1,2$
	\end{center}%$\quad$\newline
	\begin{center}
	$\ad{M_{k_2}}{N_1}=\ad{M_{k_1}}{N_2}=:N_{12}\neq\{0\}$
	\end{center}$\quad$\newline
    \begin{tabular}{p{10cm}p{3cm}}
      and all other ad-spaces are trivial. Then the Cartan matrix entry
      in the covering Nichols algebra over $G$ is
      of type $B_2$ with the split node $\tilde{k}$ the shorter root:\newline
      \begin{center}$\left(\begin{matrix}
	  C_{\tilde{k}\tilde{k}} & C_{\tilde{k}\tilde{l}}\\
	  C_{\tilde{l}\tilde{k}} & C_{\tilde{l}\tilde{l}}
	  \end{matrix}\right)
	  =\left(\begin{matrix}
	  2 & -2\\
	  -1 & 2
	\end{matrix}\right)$\end{center}$\quad$\newline
      This case can never occur isolated in indecomposable coverings, but the
      reader my check for example the ramified edge in $E_6\mapsto F_4$ in
      Section \nref{sec_RamifiedEF}.
&\hspace{.5cm}\raisebox{+.30\totalheight}{\imageplain{0.26}{DynkinA3B2}} 
      \end{tabular}
  \end{enumerate}
\end{theorem}
\vspace{-1.5cm}
\begin{proof}
  The assumed standard-Lie-type of $M$ in each case above translates
  into the knowledge of the resp. ad-spaces corresponding to the positive
  roots (see Definition  \nref{def_Cartan}). Because $\B(M)\cong \B(\tilde{M})$
  as algebra and braided vector space (see  Corollary
  \nref{cor_braidingIsomorphic}), we can hence directly calculate  case-by-case
  the respective ad-spaces of $\tilde{M}$ and hence the  nondiagonal Cartan
  matrix
  $$\tilde{C}_{\tilde{k}\tilde{l}}:=
  -\max_m\left(\adn{m}{\tilde{M}_i}{\tilde{M}_j}\neq \{0\}\right)$$
  Here the new simple summands $\tilde{M}_i$ of dimension $1$ or $2$ determined
  in Lemma \nref{lm_SplitInert} and the calculations take place in $\B(M)$.\\

  Note that in contrast to the proceeding e.g. in \cite{HV13} we do not
determine a Nichols  algebra over ad-spaces as Yetter-Drinfel'd modules in
$\B(\tilde{M})$, but  merely piece together the assumed ad-spaces in $\B(M)$.

  \begin{enumerate}
    \item In this case all ad-spaces are trivial:
      \begin{align*}
	\ad{\tilde{M}_i}{\tilde{M}_j}
	&=\ad{\bigoplus_{k\in\tilde{k}}M_{k}}
	{\bigoplus_{l\in\tilde{l}}M_{l}}\\
	&=\sum_{k,l} \ad{M_{k}}{M_{l}}\\
	&=\{0\}
      \end{align*}
    \item If both $\tilde{k}=\{k\}$ and $\tilde{l}=\{l\}$ are inert, then
    $M_{\tilde{l}},M_{\tilde{k}}$ are equal to $M_{l},M_k$ and, using again
    that $\B(M)=\B(\tilde{M})$, all ad-spaces coincide.
    \item We calculate the ad-spaces of $\tilde{M}$ over $G$
      from the assumed type $A_2\times A_2$ of $M$ over $\Gamma$, i.e. both
      $N_i$ have trivial ${\rm ad}(N_i)$:
      \begin{align*}
	\ad{\tilde{M}_{\tilde{k}}}{\tilde{M}_{\tilde{l}}}
	&=\ad{M_{k_1}\oplus M_{k_2}}{M_{l_1}\oplus M_{l_2}}\\
	&=\ad{M_{k_1}}{M_{l_1}}
	+\ad{M_{k_1}}{M_{l_2}}\\
	&\quad+\ad{M_{k_2}}{M_{l_1}}
	+\ad{M_{k_2}}{M_{l_2}}\\
	&=N_{1}\oplus N_{2}\\
	\adn{2}{\tilde{M}_{\tilde{k}}}{\tilde{M}_{\tilde{l}}}
	&=\ad{M_{k_1}\oplus M_{k_2}}{N_1\oplus N_2}\\
	&=\{0\}
      \end{align*}
    \item We calculate the ad-spaces of $\tilde{M}$ over $G$
      from the assumed type $A_3$ of $M$ over $\Gamma$, i.e. the
      $N_{12}$ below has trivial ${\rm ad}(N_{12})$:
      \begin{align*}
	\ad{\tilde{M}_{\tilde{k}}}{\tilde{M}_{\tilde{l}}}
	&=\ad{M_{k_1}\oplus M_{k_2}}{M_{l}}\\
	&=N_{1}\oplus N_{2}\\
	\adn{2}{\tilde{M}_{\tilde{k}}}
	{\tilde{M}_{\tilde{l}}}
	&=\ad{M_{k_1}\oplus M_{k_2}}{N_1\oplus N_2}\\
	&=N_{12}\\
	\adn{3}{\tilde{M}_{\tilde{k}}}{\tilde{M}_{\tilde{l}}}
	&=\ad{M_{k_1}\oplus M_{k_2}}{N_{12}}\\
	&=\{0\}\\
	\adn{2}{\tilde{M}_{\tilde{l}}}{\tilde{M}_{\tilde{k}}}
	&=\ad{M_{l}}{N_1\oplus N_2}\\
	&=\{0\}\\
      \end{align*}
  \end{enumerate}
\end{proof}

\newpage

\subsection{Example: Folding \texorpdfstring
{$A_2\times A_2$ to $A_2$ over the group $\D_4$}
{A2xA2 to A2 over the group D4}
}\label{sec_D4}

\begin{narrow}{-.7cm}{0cm}
\begin{tabular}{p{10cm}p{3cm}}      
In \cite{MS00} Milinski and Schneider gave examples of indecomposable Nichols
algebras over the non-abelian Coxeter groups
$G=\D_4,\S_3,\S_4,\S_5$. We want to show how the first case $G=\D_4$ may be
constructed as a covering Nichols algebra $\B(\tilde{M})$ of a certain
diagonal $\B(M)$ with $q$-diagram
$A_2\times A_2$ over
$\Gamma=\Z_2\times \Z_2$.\newline
$\quad$\newline
The example shall demonstrate the systematic approach in this article and
exhibits
already several crucial points of interest:

  \begin{itemize}
        \item The construction of an $M$ with a twisted symmetry of order $2$ is
      a
      model for the so-called ``unramified case'' of the covering construction
      in      Section
      \nref{sec_Unramified}: The diagram consists of disconnected
      copies of a $q$-diagram interchanged by the twisted symmetry.
\end{itemize}
&\hspace{-.1cm}
\raisebox{-.10\totalheight}{\qquad\imageplain{0.26}{DynkinA2A2A2}} 
\end{tabular}
\end{narrow}
\vspace{-.5cm}
\begin{itemize}
\item As
      described in \cite{MS00} p. 21, Gra\~na had remarked, that
      the Nichols algebra possesses a
      ``strange'' alternative basis that is non-homogeneous in the
      $G$-grading, but of precise type $A_2\times A_2$. This
      allowed Schneider and Milinski to more easily write down the relations.
      The covering
      construction precisely reproduces this basis as the formerly homogeneous
      basis of $M$ that is no longer homogeneous in $\tilde{M}$. This existence
      of a finer diagonal PBW-basis continues throughout this article (only
      in the unramified cases this finer PBW-basis type
      $X_n\times X_n$ as here, ramified cases are more involved).
    \item The Nichols algebra itself is non-faithful and even diagonal. On the
      other hand, there is a Doi twist of $\B(M)$, which is a faithful
      indecomposable Nichols algebra over $\D_4$. We will produce faithful Doi
      twists as well in Section \nref{sec_NondiagonalExamples}.
  \end{itemize}

\begin{example}
We start with a specific $4$-dimensional
diagonal \ydm over $\Gamma:=\Z_2^2=\langle v,w\rangle$ (for a systematic
construction see Section \nref{sec_Unramified})
$$M=\bigoplus_{i=0}^4\O_{\bar{g}_i}^{\chi_i}=\bigoplus_{i=0}^4 y_i\k$$
$$\bar{g}_1=\bar{g}_3=v\qquad \bar{g}_2=\bar{g}_4=w$$
$$\chi_1=\chi_4=(-1,-1)\qquad\chi_3=(-1,+1) \qquad \chi_2=(+1,-1)$$
where the tuples denote the character value on the generators:
$\chi=(\chi(v),\chi(w))$. According to \cite{Heck09}, this $M$ is of type
$A_2\times A_2$ and hence the Hilbert series of the Nichols algebra $\B(M)$ is
as follows (denoting $[n]_t:=\frac{t^n-1}{t-1}$ and esp. $[2]_t=1+t$) 
$$\H(t)=\left((1+t)(1+t)(1+t^2)\right)^2=[2]^4_t [2]^2_{t^2}
\qquad \dim\left(\B(M)\right)=\H(1)=2^6=64$$
The covering Nichols algebra $\B(\tilde{M})$ over the $\Z_2$-stem-extension
$G=\D_4$ of $\Gamma$ is of type $A_2$ with nodes of dimension $2$. But since
$\B(\tilde{M})\cong\B(M)$, it is still a diagonal, with now
non-homogeneous diagonal $\tilde{M}$-basis
$y_1,y_2,y_3,y_4$, has the same Hilbert series and dimension and a finer
PBW-basis of type $A_2\times A_2$.
\end{example}

More precisely, the construction proceeds step-by step as follows:\\

Let $a^4=b^2=1$ be the usual generators of $\D_4$. We choose a splitting
$s:\D_4\rightarrow \Gamma$ by sending the elements $1,v,w,vw$ to
$1,b,ab,a^3=bab$. The group-2-cocycle $\sigma\in Z^2(\Gamma,\Sigma^*)$ and
especially the evaluation on the generator $\theta\in\Z_2\cong \Sigma$ can hence
be calculated explicitly (rows, columns are labeled $1,v,w,vw$):
$$\sigma_\theta=\begin{pmatrix} 	
	1 & 1 & 1 & 1 	\\ 
	1 & 1 & 1 & 1 	\\ 
	1 & -1 & 1 & -1 \\ 
	1 & -1 & 1 & -1 \\ 
\end{pmatrix}$$
To apply the Construction Theorem \nref{thm_Construction} we
need an action of $\Sigma$ on $M$ by twisted symmetries with respect to
$\sigma$. Hence we first calculate from $\sigma$ the nontrivial
bimultiplicative form $\langle\rangle_\theta$ in Definition
\nref{def_TwistedSymmetryWithRespect}:
\begin{align*}
 \langle-,v\rangle_\theta&=\sigma_\theta(-,v)\sigma_\theta^{-1}(v,-)
=(+1,-1)\\
 \langle-,w\rangle_\theta&=\sigma_\theta(-,w)\sigma_\theta^{-1}(w,-)
=(-1,+1)
\end{align*}
This form immediately determines the twisted characters by Corollary
\nref{cor_twistedAction}:
\begin{align*}
  \chi_1^{\sigma_\theta}(-)&=(+1,-1)\chi_1(-)=\chi_3(-) \\
  \chi_3^{\sigma_\theta}(-)&=(+1,-1)\chi_3(-)=\chi_1(-) \\
  \chi_2^{\sigma_\theta}(-)&=(-1,+1)\chi_2(-)=\chi_4(-) \\
  \chi_4^{\sigma_\theta}(-)&=(-1,+1)\chi_4(-)=\chi_2(-) 
\end{align*}
Hence switching $y_1,y_3$ respectively $y_2,y_4$ is a twisted symmetry with
respect to the bimultiplicative form $\langle-,-\rangle_\theta$ (Definition
\nref{def_TwistedSymmetryWithRespect}). Moreover, taking this map as $f_\theta$
(and $f_1:=id$) defines an action of $\Sigma$ on $M$ by twisted symmetries with
respect to the cocycle $\sigma$ (Definition
\nref{def_TwistedSymmetryGroup}). The covering
construction (Theorem \nref{thm_Construction}) hence yields an indecomposable
Nichols algebra of dimension
$\dim\B(\tilde{M})=\dim\B(M)=64$ over $G=\D_4$. \\

To connect to the notation in \cite{MS00} we now also calculate the
$G$-homogeneous components, as they follow from the construction theorem as
$f_p$-eigenvectors to the
trivial eigenvalue $1^*\in\Sigma^*$ with $1^*(\theta)=1$ or the unique
nontrivial eigenvalue $\theta^*\in\Sigma^*$ with $\theta^*(\theta)=-1$:
\begin{align*}
x_1&:=y_1+y_3\in M_{v}^{[1^*]}=\tilde{M}_{b}\\
x_2&:=y_2+y_4\in M_{w}^{[1^*]}=\tilde{M}_{ab}\\
x_3&:=y_1-y_3\in M_{v}^{[\theta^*]}=\tilde{M}_{1,a^2b}\\
x_4&:=y_2-y_4\in M_{w}^{[\theta^*]}=\tilde{M}_{1,a^3b}
\end{align*}

\section{Symplectic Root Systems}
\subsection{Symplectic \texorpdfstring
{$\F_p$-Vector Spaces}
{Fp-Vector Spaces} And Stem Extensions
}\label{sec_SymplecticVector}
Suppose we are given a finite group $G$ with commutator subgroup $[G,G]=\Z_p$.
Such a group is clearly always a stem-extension of
its abelianization $\Gamma=G/[G,G]$. As usual e.g. for p-groups (see
e.g. \cite{Hup83}) we consider the \emph{commutator map $[,]$}, which is
skew-symmetric and isotropic:
$$G\times G\stackrel{[,]}{\longrightarrow} [G,G]=\Z_p$$
$$g,h\mapsto [g,h]=ghg^{-1}h^{-1}$$
$$[h,g]=[g,h]^{-1} \quad [g,g]=1$$
Because $[G,G]$ is central, the map is
multiplicative in both arguments:
\begin{align*}
  [g,h][g',h]
  &=(ghg^{-1}h^{-1})(g'hg'^{-1}h^{-1})\\
  &=g(g'hg'^{-1}h^{-1})hg^{-1}h^{-1}\\
  &=gg'hg'^{-1}g^{-1}h^{-1}\\
  &=[gg',h]
\end{align*}
and factors to $\langle,\rangle:\;\Gamma\times\Gamma\rightarrow \Z_p$. Because
of bimultiplicativity, $[g^p,h]=[g,h]^p=1$ holds and thus the commutator map
even factorizes one step further to $V:=\Gamma/\Gamma^p\cong \F_p^n$ 
$$V\times V \stackrel{\langle,\rangle}{\longrightarrow} \F_p\qquad
\text{{\bf denoted additively}}$$

\begin{remark}
  Note that by claim 2 in the proof of Theorem
  \nref{thm_Construction}, this  bimultiplicative form coincides with the form
  $$\langle\bar{g},\bar{h}\rangle_\theta:=\sigma(\bar{g},\bar{h})(\theta)
    \sigma^{-1}(\bar{h},\bar{g})(\theta)$$
  associated by Definition \nref{def_TwistedSymmetryGroup} to any 2-cocycle
  $\sigma$ representing the present stem-extension 
  $$\langle\theta\rangle\cong\Z_p=\Sigma\rightarrow G\rightarrow \Gamma$$
  Thus, the form $\langle,\rangle$ determines directly the notion of twisted
  symmetry in this situation.
\end{remark}

\begin{theorem}[Burnside Basis Theorem, \cite{Hup83} Thm.
3.15 p. 273f]\label{thm_Burnside}
  For $|G|=p^N$ every minimal generating set of $G$ corresponds to a
  $\F_p$-basis in the quotient $V:=G/(G^p[G,G])$, where $G^p[G,G]$ is the
  Frattini subgroup of $G$. Especially every minimal generating set consists
  precisely of $n=\dim_{\F_p}(V)$ elements. Note that for $p=2$ we have
  $V=G/(G^2[G,G])=G/G^2$.
\end{theorem}
% \begin{proof}
%   Take a set $\{g_1,\ldots g_n\}$ with their images $\bar{g}_i$ forming a
% basis
%   of $V=\Gamma/p\Gamma$, then obviously the images also generate
%   $G$; because some $g_i,g_j$ ought to be discommuting (otherwise
%   $[G,G]=1$). Also, no element may be omitted,
%   otherwise
%   the remaining images could not generate all of $V$. But the images of a
%   generating set of $G$ certainly have to generate the quotient $V$. Hence,
%   the $g_i$ form a minimally generating set.\\
% 
%   Assume conversely some set $\{g_1,\ldots g_k\}$ to minimally generate
%   $G$:
%   Again, the images in $V$ generate the entire quotient $V$. Assumed some
% linear
%   dependency, one may omit an element $g_l$ without compromising the
% generation
%   of all $V$ and (as shown above) the remaining $g_i$ still generate the
% entire
%   group. Thus the images form a basis.
% \end{proof}

In what follows, we shall consider $V=G/([G,G]G^p)$ as a \emph{symplectic
vector space} $\F_p^n$ with (possibly degenerate!) \emph{symplectic form}
$\langle v,w\rangle$. For a sub-vector space
$W\subset V$ we define the \emph{orthogonal complement}: 
$$W^\perp:=\{v\in V\;|\;\forall_{w\in W}\;\langle v,w \rangle=0\}$$
Especially $V^\perp=Z(G)/([G,G]G^p)$ is the \emph{nullspace}\index{Nullspace
$V^\perp$
(Symplectic)} of vectors orthogonal on all of $V$ (note that always $\langle
v,v\rangle=0$). For $V^\perp=\{0\}$ we call
$V$ \emph{nondegenerate}\index{Nondegenerate (Symplectic)}.

 It is well known (see e.g.
\cite{Hup83}) that there is always a \emph{symplectic basis}\index{Symplectic
Basis} $\{x_i,y_i,z_j\}_{i,j}$
consisting of mutually orthogonal {nullvectors} $z_j\in V^\perp$ and
\emph{symplectic base pairs} $\langle x_i,y_i\rangle=1$ generating a maximal
nondegenerate subspace. Note especially, that nondegenerate symplectic
vector spaces hence always have even dimension! They lead
for example to extraspecial groups $G=p_\pm^{\dim(V)+1}$\index{Extraspecial
group $p_\pm^{2n+1}$}, especially for $p=2$ and $\dim(V)=2$ to $G=\D_4,\Q_8$.

\subsection{Symplectic Root Systems Of Type \texorpdfstring{$ADE$}{ADE} Over
\texorpdfstring {$\F_2$}{F2}
}\label{sec_SymplecticRootsystems}
 
\begin{definition}\label{def_SymplecticRootsystem}
  Given a symplectic vector space $V$ over $\F_2$ and a graph $\mathcal{D}$, we
  define a \emph{symplectic root system}\index{Symplectic root system} for this
  graph as a decoration
  $\phi:Nodes(\mathcal{D})\rightarrow V$,
  such that $\Im(\phi)$ generates $V$ and nodes $i\neq j$ are connected
  iff $\langle \phi(i),\phi(j) \rangle=1_{\F_2}$ (note that always $\langle
  v,v\rangle=0$). If $\Im(\phi)$ is even a $\F_2$-basis of $V$, we call the
  symplectic root system \emph{minimal}.
\end{definition}

\begin{remark}
We will use the notion for one directly on simply-laced Dynkin diagrams
$\mathcal{D}$, but also as tool for the ramified case $\mathcal{D}'$, where only
a part of
the diagram $\mathcal{D}\subset\mathcal{D}'$ is split (such
as $\mathcal{D}=A_2,A_{n-1}$ for ramified $E_6\mapsto F_4$ and
$A_{2n-1}\mapsto C_n)$. Minimal symplectic root systems thereby correspond
to minimally indecomposable covering Nichols algebras.
\end{remark}

Note that in \cite{Len13b} we completely classified symplectic root
systems for arbitrary given graphs over the field $\F_2$. In \cite{Len13b} Cor.
5.9 we have proven that every graph admits a unique minimal symplectic root
systems. We have also given explicit descriptions, given below, of the
symplectic root systems for Cartan type Dynkin diagrams. However, that the
reader may directly verify that the following decorations indeed do form
symplectic root systems for all simply-laced Dynkin diagrams.

\begin{theorem}\label{thm_SymplecticRootsystem}
  The graph of a simply laced Dynkin diagram of rank $n$ 
  admits a minimal symplectic root system over the symplectic vector space $V$
  of dimension $n$ and typically minimal nullspace dimension $k=\dim(V^\perp)$: 
  \begin{itemize}
    \item $k=0$ for $n$ even, i.e. $V=\langle \{ x_i,y_i\}_i\rangle_\k$
    \item $k=1$ for $n$ odd, i.e. $V=\langle\{ x_i,y_i\}_i,z\rangle_\k$
    \item $k=2$ for type $D_n$ and $n$ even, i.e. 
      $V=\langle \{x_i,y_i\}_i,z_1,z_2\rangle_\k$.
  \end{itemize}
  \end{theorem}
  \begin{proof}
  Explicit symplectic root systems are given by the following decorations
  $\phi$. 
  One checks easily in every
  instance, that indeed $\langle \phi(i),\phi(j)\rangle=1$ iff $i,j$ are
  connected:  
  
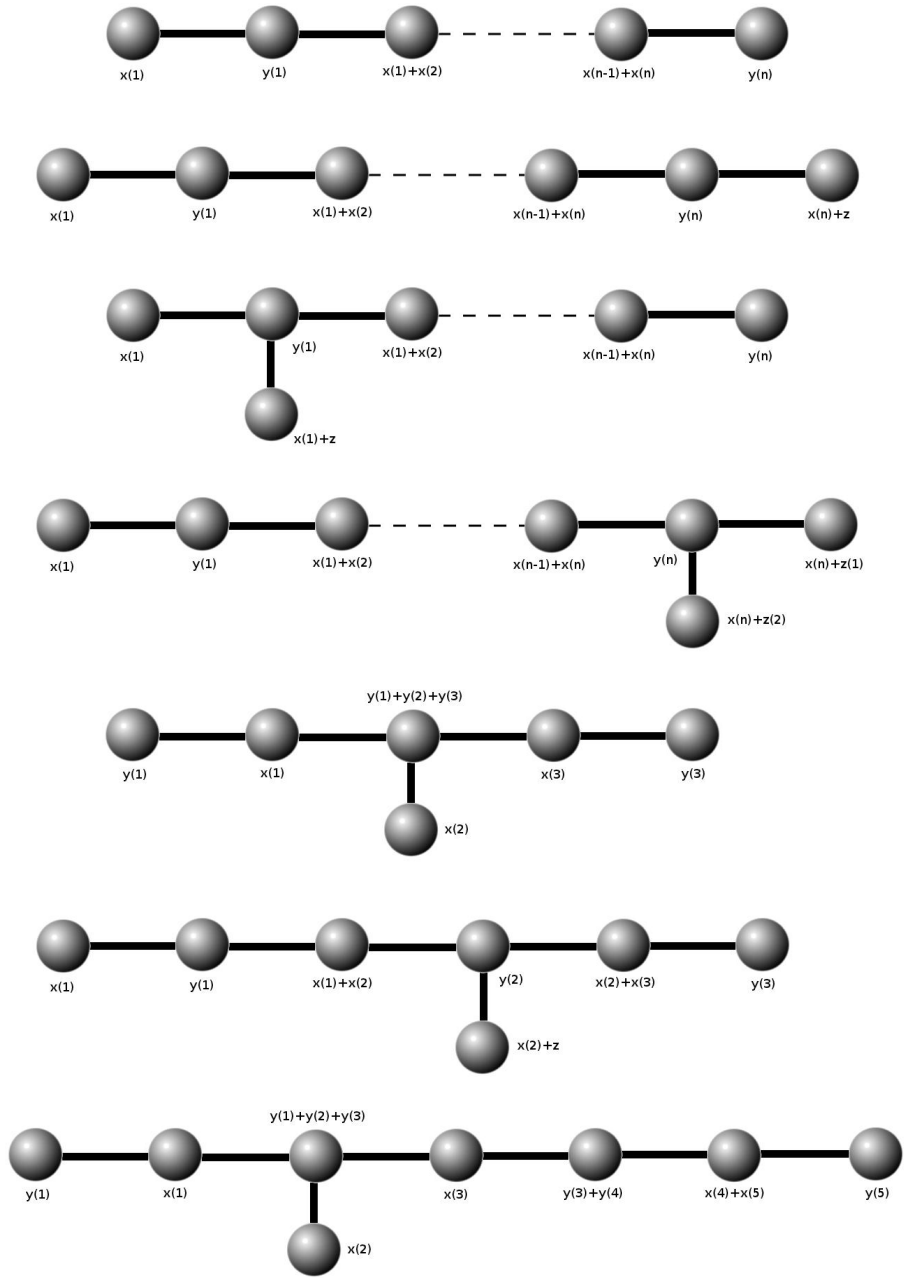
\begin{center}

\enlargethispage{2cm}

%    $$A_{n},\;2|n$$
    \imageplain{0.26}{SymplecticRootA2n}
  %\end{center}
  %\begin{center}
%    $$A_n,\;2\not|n$$
    \imageplain{0.26}{SymplecticRootA2n1}
  %\end{center}
  %\begin{center}
%    $$D_n,\;2\not|n$$
    \imageplain{0.26}{SymplecticRootD2n1}
  %\end{center}
  %\begin{center}
%    $$D_n,\;2|n$$
    \imageplain{0.26}{SymplecticRootD2n2z}
  %\end{center}
%
  %\begin{center}
%    $$E_6$$
    \imageplain{0.26}{SymplecticRootE6}
  %\end{center}
  %\begin{center}
%
%    $$E_7$$
    \imageplain{0.26}{SymplecticRootE7}
  %\end{center}
  %\begin{center}
%    $$E_8$$
    \imageplain{0.26}{SymplecticRootE8}
  \end{center}

%   \begin{center}
%     \begin{tabular}{r|ccccccc}
%       $M=$ & $A_{2n}$ & $A_{2n-1}$ & $D_{2n+3}$ & $D_{2n+2}$ 
% 	& $E_6$ & $E_7$ & $E_8$ \\  
%       \hline
%       $dim(V^\perp)=$ & $0$ & $1$ & $1$ & $0,2$ & $0$ & $1$ & $0$ \\
%     \end{tabular}
%   \end{center}

\end{proof}

\section{Main Constructions For Commutator Subgroup \texorpdfstring
{$\Z_2$}
{Z2}
}\label{sec_mainConstruction}
Suppose a nonabelian group $G$ with commutator subgroup
$[G,G]=\Z_2$, which is hence a stem-extension $\Sigma^*=\Z_2\rightarrow
G\rightarrow \Gamma$ of an {abelian} group $\Gamma$. Using Heckenberger's
classification \cite{Heck09} of finite-dimensional Nichols algebras $\B(M)$ over
abelian groups $\Gamma$ and symplectic root systems, we now construct
finite-dimensional minimally indecomposable covering Nichols algebras
$\B(\tilde{M})$ with connected Dynkin diagram, depending on 2-rank and
2-center of $G$. 

We need the following additional technical assumption to present the main
result, that gives us control over the size of even-order generating systems in
a group. Groups that fail this assumption can easily be treated if explicit
generating systems are at hand, see e.g. Corollary \nref{cor_Nonsaturated}.
Usually such a group still admits {\bf dis}connected covering Nichols algebras.

\begin{definition}\label{def_Saturated}
  A nilpotent group $G$ is \emph{$2$-saturated}, if for any prime $p$
  $$\dim_{\F_p}(G/[G,G]G^p)\leq \dim_{\F_2}(G/[G,G]G^2)
    \qquad\left(=\dim_{\F_2}(G/G^2)\right)$$
  Especially every $2$-group is $2$-saturated.
\end{definition}
\begin{lemma}\label{lm_Saturated}
  If a nilpotent group $G$ is $2$-saturated, then any $\F_2$-basis of the
  elementary abelian quotient $V=G/G^2$ can be lifted to a minimally generating
  set of $G$. Especially any minimal generating set of $G$ contains only
  elements of even order and has size $\dim_{\F_2}(G/G^2)$.
\end{lemma}
\begin{proof}
  By \cite{Hup83} Thm 2.3 (p. 260f) every nilpotent group $G$ is a direct
  product  of $p$-groups $G_p$, i.e. groups of prime power. By the Burnside
  Basis Theorem \nref{thm_Burnside} a minimal generating set of any $G_p$
  corresponds to a $\F_p$-basis of the elementary abelian quotient
  $G_/p[G_p,G_p]G_p^p\cong\F_p^{N_p}$ with $N_p:=\dim_{\F_p}(G/[G,G]G^p)$. The
  group $G$ is 2-saturated, iff  $N_2$ is a maximal value of all $N_p$. Hence
  under this assumption any $\F_2$-basis of 
  $$V=G_2/([G_2,G_2]G_2^2)=G_2/G_2^2\cong G/G^2$$
  hence of rank $N_2$, allows for the choice of generating set in the
  other  $G_p$ as well. The respective direct products form a generating set in
  $G$. Omitting one element would omitt an element in the $\F_2$-basis, which
  can hence not be generating any more. Thus the constructed generating set is
  minimal.
\end{proof}

In Section \nref{sec_NondiagonalExamples} we will furthermore
give nondiagonal (and especially some faithful!) Doi twists of the diagonal
Nichols algebras constructed below. These examples of rank $\leq 4$ over
various groups of order $16$ and $32$ are then the first indecomposable faithful
Nichols algebras over nonabelian groups of rank $>2$.\\

\begin{theorem}\label{thm_AllZpOrbifolds}
  For any group $G$ with $[G,G]\cong\Z_2$ consider the
  invariants
  \begin{center}
  \begin{tabular}{lll}
   $V$ & $:=G/G^2\;\cong\Gamma/\Gamma^2$ & 
    $\qquad\dim_{\F_2}(V)=:$ 2-rank\\
   $V^\perp$ & $=Z(G)/G^2$ & 
    $\qquad\dim_{\F_2}(V^\perp)=:$ 2-center\\
  \end{tabular}
  \end{center}
  and denote $(n\mod 2)\in\{0,1\}$. Assume $G$ to be
  $2$-saturated, then $G$ admits a finite-dimensional
  minimally indecomposable Nichols algebra $\B(\tilde{M})$ with the following
  connected Dynkin diagram, depending on 2-rank and 2-center of $G$. They
  are covering Nichols algebras of some $\B(M)$ over $\Gamma$ constructed
  below:
  \begin{itemize} 
      \item {\bf Unramified} (generic) simply-laced components from a
	disconnected double with a symplectic root system for $V$:
      \setlength{\tabcolsep}{12pt}
      \begin{narrow}{1.0cm}{0cm}
      \begin{tabular}{cr|cc}
	$dim_{\F_2}(V)$ & $dim_{\F_2}(V^\perp)$
	  & $M$ & $\tilde{M}$\\
	\hline
	$n$ & $n \mod 2$ 
	  & $A_{n}\times A_{n}$ & $A_{n\geq 2}$\\
	$n$ &  $n\mod 2$ 
	  & $E_n\times E_n$ & $E_{n=6,7,8}$ \\
	$n$ & $2-(n \mod 2)$ 
	  & $D_{n}\times D_{n}$ & $D_{n\geq 4}$\\
      \end{tabular}
      \end{narrow}
       \setlength{\tabcolsep}{6pt}
    \item {\bf Ramified} components from a single diagram with
      an order 2 automorphism and a symplectic root system decomposed as
      $V:=V_{inert}\oplus^\perp V_{split}$:
      \setlength{\tabcolsep}{12pt}
      \begin{narrow}{-0cm}{0cm}
      \begin{tabular}{rlrl|cc}
	\multicolumn{2}{c}{$dim_{\F_2}(V)$}
	& $dim_{\F_2}(V^\perp)$
	& $M$ & $\tilde{M}$\\
	\hline
	$2+2$ & $=4$ & $2$ 
	  & $E_6$ & $F_4$\\
	$1+(n-1)$ & $=n$ & $2-(n\mod 2)$
	  & $A_{2n-1}$ & $C_{n\geq 3}$\\
      \end{tabular}\\
      \end{narrow}
      \setlength{\tabcolsep}{6pt}
  \end{itemize}
  \end{theorem}
  \begin{proof}
    The theorem states a list of constructions, provided that $G$ is of the
    assumed form with specific invariants $dim_{\F_2}(V)$ and
    $dim_{\F_2}(V^\perp)$. These construction are then carried out in the
    subsequent sections:
    \begin{itemize}
      \item The unramified cases in Theorem \nref{thm_Unramified}.
      \item The ramified case $E_6\mapsto F_4$ in Theorem \nref{thm_RamifiedEF}.
      \item The ramified case $A_{2n-1}\mapsto C_n$ in Theorem
	\nref{thm_RamifiedAB}.
    \end{itemize}	
    In each of the quoted theorems, the assumed invariants of $G$ characterize
    the type (dimension, nullspace) of the symplectic vector space $G/G^2$ as
    described in Section \nref{sec_SymplecticVector}. We then invoke Theorem
    \nref{thm_SymplecticRootsystem}, which returns  a minimal
    symplectic root system precisely for the distinct choice of invariants
    assumed in the statement. This basis of $V=G/G^2$ is then lifted using
    Lemma \nref{lm_Saturated} using the additional assumption of $G$ to be
    saturated. This yields a set of conjugacy classes, which we complete to a
    Yetter-Drinfeld modules $M$ with twisted symmetry as prescribed
    by an ad-hoc choice of realizing characters. This finally stages the
    application of the covering construction Theorem \nref{thm_Construction} to
    yield $\B(\tilde{M})$.
\end{proof}
 
\begin{remark}
  For further link-{\bf de}composable Nichols algebras see
\cite{Len12}  Sec.  6.6:
  \begin{itemize}
    \item unramified $A_1\times A_1\mapsto A_1$
    \item ramified $A_3\mapsto B_2$, $D_{n+1}\mapsto B_n$ and
      several alike non-Cartan diagrams.
    \item an isolated loop diagrams $A_2\mapsto A_1,\;q\in \k_3$. 
    \item ramified $D_4\rightarrow G_2$ which is the {only}
      covering with $\Sigma\cong \Z_3$. 
  \end{itemize} 
\end{remark}

Before we proceed to the proof, we comment on possible outlooks to groups of
higher nilpotency class:

\begin{remark}\label{rem_class3}
  Note that non-minimally indecomposable covering Nichols algebras over $G$
  might be interesting as well, especially because in this case there might
  exist a secondary covering Nichols algebra $\B(\tilde{\tilde{M}})$ over a
  group $\tilde{G}$ of nilpotency class $3$, that use a twisted symmetry of the
  primary covering Nichols algebra $\B(\tilde{M})$ over $G$.\\

  In \cite{Len13b} we have found non-minimal symplectic root systems as well,
  especially of type $D_{n+1}$. As a conjecture, this might give rise to
  covering Nichols algebras 
  $$D_{n+1}\times D_{n+1} \mapsto D_{n+1} \mapsto B_{n}$$
  over a group $\tilde{G}$ of nilpotency class $3$. In this
  case, $\tilde{\tilde{M}}$ would be the sum of a unique simple \ydm of
  dimension $4$ and $n-1$ simple \ydms of dimension $2$. Such a Nichols algebra
  $\B(\tilde{\tilde{M}})$ would exhibit a $B_n$-Dynkin diagram, while the
  Hilbert series were the square of a Hilbert series of a Nichols
  algebra with diagram $D_{n+1}$ over an abelian group $\Gamma$.\\

  During the review of this article, such a Nichols algebra for $n=2$
  has indeed be discovered in the classification of rank 2 in \cite{HV13}.
  However, the assumed Hilbert series
  $\H(t)=\left([2]_t^3[2]_{t^2}^2[2]_{t^3}\right)^2$
  turns out to be just a large divisor of the found Hilbert series, which
  indicates an additional extension. We hope nevertheless that the present
  approach can help to further understand the structure of the newly discovered
  Nichols algebra, such as providing a PBW-basis.
\end{remark}

\subsection{Unramified Cases \texorpdfstring
{$ADE\times ADE\mapsto ADE$}
{ADE x ADE -> ADE}
}\label{sec_Unramified}

The most natural and generic way to construct a \ydm with twisted symmetry
$\Z_2$ has already been demonstrated on the case $\D_4$
in Section \nref{sec_D4}; we take $\Gamma$-\ydms $N$ as well as $N_\sigma$
with modified $\Gamma$-action (Remark \nref{rem_twistedAction}) and force
twisted symmetry by considering $M:=N\oplus N_\sigma$. The symplectic root
system will assures that the diagram of $M$ consists indeed of disconnected
identical subdiagrams for $N,N_\sigma$ (while $N\not\cong N_\sigma$). We will
subsequently calculate an explicit example for $A_4\times A_4\mapsto A_4$ in
Section \nref{sec_exampleA4}.\\

The following image illustrates the covering of type $D_n\times D_n\mapsto D_n$:

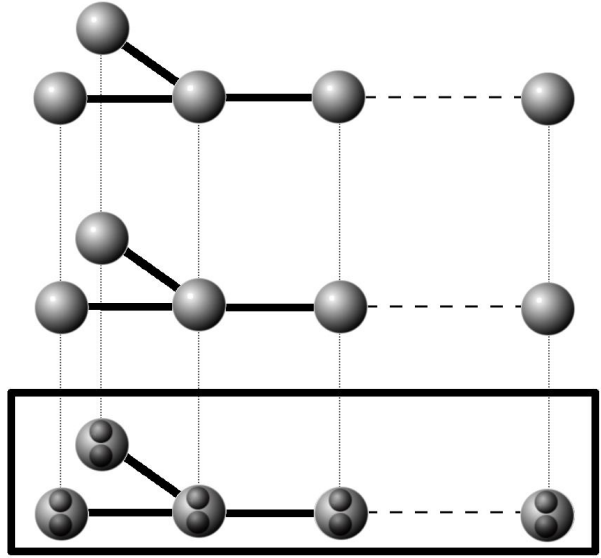
\image{0.26}{DynkinDnDn}

\begin{theorem}\label{thm_Unramified}
  Suppose a simply-laced Dynkin diagram $X_n$ of rank $n$ and $G$ an arbitrary
  $2$-saturated group (Def \nref{def_Saturated}) with $[G,G]=\Z_2$ and
  $\Gamma:=G/[G,G]$, such that 
  \begin{itemize}
    \item $dim_{\F_2}(V)=dim_{\F_2}(G/G^2)\stackrel{!}{=}n\geq 2$
    \item
      $dim_{\F_2}(V^\perp)=dim_{\F_2}(Z(G)/G^2)\stackrel{!}{=}
      \begin{cases}
       2,\text{ for }X_n=D_{2m} \\ n\mod 2,\text{ else }
      \end{cases}$
  \end{itemize}
  Then there exists a $\Gamma$-\ydm $M=N\oplus N_\sigma$ with $N,N_\sigma$
  disconnected in $M$ and an twisted permutation action of $\Sigma\cong\Z_2$ 
  interchanging
  $N\leftrightarrow N_\sigma$. The covering \ydm over $G$ is hence
  $\tilde{M}=\bigoplus_{i=1}^n\tilde{M}_i$ of dimension $2n$ with: 
  \begin{itemize}
   \item $[G,G]$ acts trivially on $\tilde{M}$, which is
	 diagonal, but $V$ acts faithfully.
   \item $\tilde{M}$ is minimally indecomposable.
   \item $\B(\tilde{M})$ is finite-dimensional, with Hilbert series the
	  square of the Hilbert series single diagram in the
	 diagonal case over $\Z_2^n$, especially 
	$$\dim\left(\B(\tilde{M})\right)=\H(1)=2^{|\Phi^+(X_n\times X_n)|}$$
   \item $\tilde{M}$ has the prescribed Cartan matrix and Dynkin diagram
	  $X_n$ with
	  all nodes $\tilde{M}_i$ of dimension $2$ (i.e. underlying conjugacy
	  class of length 2).
  \end{itemize}
  An explicit example of type $A_4$ will be discussed in the subsequent
  subsection \nref{sec_exampleA4}. Several faithful Doi twist and hence
  nondiagonal Nichols algebras for small
  rank $D_4,A_2,A_3$ over nonabelian $G$ are given in Section
  \nref{sec_NondiagonalExamples}.
\end{theorem}
\begin{proof}The strategy has been outlined above:\\

  {\bf Step 1:} We first construct a $\Gamma$-\ydm $N:=\bigoplus_{i=1}^n
\O_{\bar{g}_i}^{\chi_i}$,
  such that 
  \begin{itemize}
    \item $N$ is minimally indecomposable
    \item The braiding matrix only contains $\pm1$
    \item The quotient $V$ acts faithfully
    \item Nodes $i,j$ are connected iff $\langle\bar{g}_i,\bar{g}_j\rangle \neq
      0$ (i.e. any lifts $g_i,g_j\in G$ discommute)
    \item The Nichols algebra is finite-dimensional and has the prescribed
    Dynkin diagram
  \end{itemize}

  This is done by using precisely the symplectic root systems constructed in
  Sections \nref{sec_SymplecticVector}-\nref{sec_SymplecticRootsystems}:
  $V:=G/G^2$ is a symplectic
  vector space as described in the cited Section with dimension 
  $dim_{\F_2}(G/G^2)$ and nullspace dimension 
  $dim_{\F_2}(Z(G)/G^2)$. Hence the assumptions of the
  present theorem
  exactly match those of {\it cit. loc.} and we get a symplectic root system
  basis $\phi(i)$ ($1\leq i\leq n$) of $V$,
  i.e. $\langle\phi(i),\phi(j)\rangle\neq 0$ iff $i,j$ are connected.
  Because $G$ was assumed $2$-saturated, this basis of $V$ can be lifted to a
  minimally generating set $g_i$ of $G$.

  We define an indecomposable \ydm $N$ by using the images of the
  minimally generating set $\bar{g}_i$ of $\Gamma=G/[G,G]$ (=coaction). 
  Then we construct suitable characters
  $\chi_i:\;\Gamma\rightarrow \k^\times$ (=action) that realize the given
  diagram with braiding matrix $\pm1$. Because
  the $\phi(i)$ were a basis of $\Gamma/\Gamma^2$, there is exactly one
  $\chi_i$ such that $\chi_i(\bar{g_j})=-1$ if $i=j$ or $i<j$ are connected and
  $+1$ otherwise. 
  Then $N:=\bigoplus_i \O_{\bar{g}_i}^{\chi_i}$ has by construction a braiding
matrix with
  monodromy $q_{ij}q_{ji}\neq 1$ precisely iff lifts $g_i,g_j$ discommute
  in $G$.  Note by construction, as $\F_2$-matrix $\chi_1,\ldots \chi_n$ is
  triangular,  hence $V$ acts faithful, which also proves this part of the
  statement.\\
 
  {\bf Step 2:} The central extension in question is
   $\Sigma^*=\Z_2\rightarrow G\rightarrow \Gamma$.
   Take a Section $s$ and $\sigma\in Z^2(\Gamma,\Sigma^*)$; we have seen during
   the proof of Theorem \nref{thm_Construction} in claim 2 that
   $$\sigma(\bar{a},\bar{b})\sigma^{-1}(\bar{b},\bar{a})=[a,b]$$
   Because the chosen generators $g_i\in G$ map to the symplectic root system
  $\phi(i)$, we know the commutators $[g_i,g_j]\in[G,G]=\Sigma^*$: Take $\theta$
  the generator of $\Sigma=\Z_2$, then the
    twisted $\Gamma$-action after applying $f_\theta$ on an
    element $v_{\bar{b}}\in M_{\bar{b}}$ reads as:
    \begin{align*}
      \bar{a}.f_\theta(v_{\bar{b}})
      &\stackrel{!}{=}\sigma_\theta(\bar{a},\bar{b})
	\sigma_\theta^{-1}(\bar{b},\bar{a})
	f_\theta(\bar{a}.{v_{\bar{b}}})\\
      &=\left(u(\bar{a},\bar{b})u^{-1}(\bar{b},\bar{a})\right)(\theta)
	f_\theta(\bar{a}.v_{\bar{b}})\\
      &=\left(\langle\bar{a},\bar{b}\rangle\right)(p)
	f_\theta(\bar{a}.v_{\bar{b}})
    \end{align*}
  Hence any decorating character on some decorating group element $\chi_k(g_l)$
  picks up an additional $-1$ iff $[g_k,g_l]\neq 1$ iff
  $\langle\bar{g_k},\bar{g_l}\rangle\neq 0$.\\

  {\bf Step 3:} We now construct a
  $\Gamma$-\ydm with an action $\Sigma=\Z_2$ by twisted permutation symmetries
  as in the
  example $\D_4$ in Section \nref{sec_D4}. We start with
  the indecomposable $N=\bigoplus_{i=1}^n N_i$ constructed in step 1. Then we
  add the necessary twisted image $f_\theta(N)$ for $\theta$ the
  generator of $\Sigma=\Z_2$ (see Remark \nref{rem_twistedAction}):
  $$N_\sigma:={N}_{\sigma_\theta}={N}_{\sigma(\theta)}$$
  $N_\sigma$ is hence the sum of simple \ydms ${N_\sigma}_i$ given by the
  same group elements $\phi(i)$ but with twisted $\Gamma$-action:
  $$\chi^{\sigma_\theta}_i(\bar{b})
    :=\left(\langle\phi(i),\bar{b}\rangle\right)(\theta)\chi_i(\bar{b})$$
  $$\bar{a}._{\sigma_\theta}v_{\bar{b}}
    =\left(\langle\phi(i),\bar{b}\rangle\right)(\theta)$$
  By construction $M:=N\oplus N_\sigma$ admits a
  twisted symmetry $f_\theta$ interchanging 
  $N_i\leftrightarrow {N_\sigma}_i$.\\
  
  {\bf Step 4:} We yet have to check that $M$ still has a finite Nichols
  algebra, so we determine its full Dynkin diagram -- as intended, we prove
  now, that it really consists of two disconnected copies of the given one.
  First be reminded on Corollary \nref{cor_diagramAutomosphisms} that twisted
  symmetries leave Dynkin diagrams and q-diagram invariant, so the diagrams of
  $N,N_\sigma$ coincide.

  Hence the tricky part is, that there are no additional mixed
  edges between any $N_i\leftrightarrow {N_\sigma}_j$, i.e.
$c_{N_i,{N_\sigma}_j}c_{{N_\sigma}_j,N_i}=id$.
  This is precisely where we need the specific base choice $\phi(i)$ to be a
  symplectic root system
      (Definition \nref{def_SymplecticRootsystem}) together with the
      fact that all $q_{ij}=\pm 1$. We
      have to calculate the mixed braiding factors:
  \begin{align*}
    q &:= q_{N_i,{N_\sigma}_j}q_{{N_\sigma}_j,N_i}\\
    &= \chi_i(\phi(j))\chi^{\sigma_\theta}_j(\phi(i))\\
    &=\chi_i(\phi(j))\cdot
    \sigma_\theta(\phi(j),\phi(i))\sigma_\theta^{-1}(\phi(i),\phi(j))
    \chi_j(\phi(i))\\
    &=\langle\phi(i),
    \phi(j)\rangle(\theta)\chi_i(\phi(j))\chi_j(\phi(i))\\
    &=\langle\phi(i),\phi(j)\rangle(\theta)q_{ij}q_{ji}
  \end{align*}
  We have to distinguish two cases that yield $q=1$ in different ways:
  \begin{itemize}
    \item Suppose $i,j$ disconnected in the original diagram. Then
      $q_{ij}q_{ji}=1$ and at the same time by construction
      $\langle\phi(i),\phi(j)\rangle=0$, hence $q=1$.
    \item Suppose $i,j$ connected by a single edge. Then
      $q_{ij}q_{ji}=-1$ and at the same time by construction
      $\langle\phi(i),\phi(j)\rangle\neq 0$, hence $=\theta^*$ for the
      generator of $\Sigma^*\cong\Z_2$ with $\theta^*(\theta)=-1$. Hence we
      again get $q=1$.\\
  \end{itemize}

  {\bf Step 5:} Thus we are done: We constructed a twist-symmetric
    indecomposable \ydm
    $M$ over $\Gamma$ with finite-dimensional Nichols algebra and Hilbert
    series $\H_M(t)=\H_{N}(t)\H_{N_\sigma}(t)=\H_{N}(t)^2$. Hence the covering
    $G$-\ydm $\tilde{M}$ is indecomposable and the Nichols algebra
    $\B(\tilde{M})$ over $G$ has Hilbert series $\H_{\tilde{M}}(t)=\H_{{M}}(t)$,
    especially it is finite-dimensional.
\end{proof}

\newpage

\subsection{Example \texorpdfstring
{$A_4\times A_4\mapsto A_4$}
{A4 x A4 -> A4}
}\label{sec_exampleA4}
  We realize $A_4$ as prescribed over a group $G$ with 2-rank
  $dim_{\F_2}(G/G^2)=4$ and no 2-center
  $dim_{\F_2}(Z(G)/G^2)=0$, such as the
  extraspecial group $G=2_+^{4+1}=\D_4\ast \D_4$ (the central product
  identifies
  the two dihedral centers), which is generated by mutually discommuting
  involutions $x,y$ and $x',y'$, corresponding to a symplectic basis of the
  nondegenerate symplectic vector space $V=\Gamma=\F_2^4$. We need a 
  $\Gamma$-\ydm of type $A_4\times A_4$ admitting an  involutory twisted
symmetry
  $$M=N\oplus N_\sigma=:(M_1\oplus M_2\oplus M_3\oplus M_4)
    \oplus (M_5\oplus M_6\oplus M_7\oplus M_8)$$
  where each $M_k=\O_{\bar{g}_k}^{\chi_k}$ is $1$-dimensional. The {group
  elements} are determined by the respective symplectic root system in Theorem
  \nref{thm_SymplecticRootsystem}:
  $$\bar{g}_1=\bar{g}_5=x \qquad \bar{g}_2=\bar{g}_6=y \qquad
    \bar{g}_3=\bar{g}_7=xx' \qquad \bar{g}_4=\bar{g}_8=y'$$
  Then the {characters} $\chi_k$ for $k\leq 4$  were defined in such a way
  that $\chi_k(\bar{g}_k)=-1$, and $\chi_k(\bar{g}_l)=-1$ for edges $k<l$ and
  $+1$ else.
  This has to be basis-transformed to be expressed as row vector showing the
  values in the original basis $(\chi(x),\chi(y),\chi(x'),\chi(y'))$:
  \begin{center}
  \begin{tabular}{lrclr}
    $\chi_1$ &  $=(-1,-1,-1,+1)$ & $\qquad$ & $\chi_2$ & $(+1,-1,-1,+1)$ \\
    $\chi_3$ &  $=(+1,+1,-1,-1)$ & $\qquad$ & $\chi_4$ & $(+1,+1,+1,-1)$ \\
  \end{tabular}
  \end{center}
  As 	calculated in general, the {twisted characters}
  $\chi_{4+k}=\chi_{k}^\sigma$ catch an additional $-1$ on every element
  $G$-discommuting with $g_k$ resp. $\not\perp\bar{g}_k$ in $V$:
  \begin{center}
  \begin{tabular}{lrclr}
    $\chi_1$ &  $=(-1,+1,-1,+1)$ & $\qquad$ & $\chi_2$ & $(-1,-1,+1,+1)$ \\
    $\chi_3$ &  $=(+1,-1,-1,+1)$ & $\qquad$ & $\chi_4$ & $(+1,+1,-1,-1)$ \\
  \end{tabular}
  \end{center} 
  Altogether we find the following covering Nichols algebra $\B(\tilde{M})$
  with Hilbert series $\H(t)$ and dimension $\dim(\B(\tilde{M}))$  as for $M$
  and hence the square of  the Hilbert  series of the prescribed diagram $A_4$
  of $M$ for $q=-1$  \cite{Heck09}:

  $$\H(t)=\left([2]^{4}_{t}[2]^{3}_{t^2}[2]^{2}_{t^3}[2]_{t^4}\right)^2
    \qquad
    \dim(\B(\tilde{M}))=\H(1)=2^{20}=2^{|\Phi^+(A_4\times A_4)|}$$
 
  \image{0.26}{DynkinA4A4Decorated}

\subsection{Ramified Case \texorpdfstring
{$E_6\mapsto F_4$}
{E6 -> F4}
}\label{sec_RamifiedEF}

The examples of the last two Sections are ``generically'' exploit a disconnected
doubling of a rather arbitrary Dynkin diagram and yield simply laced coverings.
Every
(nonabelian) edge corresponds to the $\D_4$ example above and all conjugacy
classes have same lengths.
It turns out, that the \emph{ramified case} involving different conjugacy class
lengths is far more restrictive! We shall now give an example of this type, with
$\Sigma=\Z_2$ the diagram automorphism of a single $E_6$-diagram and the
covering Nichols algebra
has non-simply-laced diagram $F_4$:
\vspace{-.5cm}
\image{0.26}{DynkinE6F4}

\begin{theorem}\label{thm_RamifiedEF}
  Suppose $G$ a $2$-saturated group with $[G,G]=\Z_2$
  and $\Gamma:=G/[G,G]$ s.t.
  \begin{itemize}
    \item $dim_{\F_2}(V)=dim_{\F_2}(G/G^2)=4$
    \item
  $dim_{\F_2}(V^\perp)=dim_{\F_2}(Z(G)/G^2)=2$
  \end{itemize}
  Then there exists a suitable $\Gamma$-\ydm $M$ of type $E_6$
  with an involutory diagram automorphisms. The covering $G$-\ydm $\tilde{M}$
  decomposes into 4 simple \ydm like
  $\tilde{M}=\bigoplus_{\tilde{k}=1}^4\tilde{M}_{\tilde{k}}$, has dimension
  $6$ and moreover: 
  \begin{itemize}
   \item $[G,G]$ acts trivially on $\tilde{M}$, which is hence
	  diagonal, but $V$ acts faithfully. 
   \item $\tilde{M}$ is minimally indecomposable.
   \item $\B(\tilde{M})$ has Hilbert series 
	 $$\H(t)=[2]^6_{t} [2]^5_{t^2} [2]^5_{t^3} [2]^5_{t^4} [2]^4_{t^5}
	[2]^3_{t^6}[2]^3_{t^7} [2]^2_{t^8} [2]_{t^9}
	[2]_{t^{10}}[2]_{t^{11}}$$
	and is thus especially of dimension 
	$$\dim\left(\B(\tilde{M})\right)=\H(1)=2^{36}=2^{|\Phi^+(E_6)|}$$
   \item $\tilde{M}$ has the Dynkin diagram $F_4$, where short
      roots correspond to conjugacy classes of length $2$ and long roots
      to a central elements.
  \end{itemize}
  In section \nref{sec_NondiagonalExamples} we give an example of a nondiagonal
  Doi twist of $\B(\tilde{M})$ over $\Z_2^2\times \D_4$.
\end{theorem}
\begin{proof}
  Denote by $\bar{z},\bar{z}',\bar{x},\bar{y}\in\Gamma$ some lifts of a
  basis of the $4$-dimensional symplectic vector space
  $V=G/G^2=\Gamma/\Gamma^2$ with
  $2$-dimensional nullspace, such that $\bar{z},\bar{z}'$ were nullvectors and
  $\bar{x},\bar{y}$ was a symplectic base pair in $V$. As $G$ was assumed to be
  $2$-saturated (Def. \nref{def_Saturated}), we may choose these lifts to be a
  minimally generating system  of $\Gamma$ as well. Because $[G,G]\cong \Z_2$,
  any further lifts to $z,z',x,y\in G$ will obey by Section
  \nref{sec_SymplecticVector}:
  $$z,z'\in Z(G) \qquad [x,y]\neq 1$$
  We directly construct the $\Gamma$-\ydm $\bigoplus_{k=1}^6
  \O_{\bar{g}_k}^{\chi_k}$ of type $E_6$, but otherwise proceed as  in the
  unramified case. Note that the following could also be derived 
  systematically using the (rather trivial) symplectic root system 
  $\bar{x},\bar{y}$ for the
  aspired  split part of $V$ and character via some ordering of the nodes, as
  it is done for the remaining ramified case below; but here we want to keep
  everything explicit! Further denote any {character} $\chi\in\Gamma^*$
  as row-vectors containing the basis images  
  $(\chi(\bar{z}),\chi(\bar{z}'),\chi(\bar{x}),\chi(\bar{y}))$, then $M$ shall 
  be (we've introduced additional signs for the faithfulness-statement):

  \image{0.26}{DynkinE6Decorated}

  One can check directly, that $q_{ii}=-1$ and the $q_{ij}q_{ji}=\pm1$ exactly
  match the given diagram. Furthermore, already $\chi_1,\chi_2,\chi_3,\chi_4$ is
  $\F_2$-linearly independent and $z,z'$ have been constructed to act as $-1$
  on $x$ resp. $y$, hence the faithfulness assertions hold. This defined a	
  proper Nichols algebra $\B(M)$ of dimension $2^{36}$ and the prescribed
  Hilbert series by \cite{HS10a} Theorem 4.5.

  We calculate now, that the $E_6$
  diagram automorphisms $f_\theta$ is here a twisted symmetry:
  $$\chi_{1}^\sigma(g_k)
    = \sigma_\theta(g_k,g_1)\sigma_\theta^{-1}(g_1,g_k)\chi_1(g_k)
    = \langle \bar{g}_k,z\rangle\chi_1(g_k)=\chi_1(g_k)$$
  \begin{align*}
    \chi_{3}^\sigma(z)
    &= \langle z,x\rangle\chi_3(z)=\chi_3(z)=+1=\chi_5(z)\\
    \chi_{3}^\sigma(z')
    &= \langle z',x\rangle\chi_3(z')=\chi_3(z')=-1=\chi_5(z')\\
    \chi_{3}^\sigma(x)
    &= \langle x,x\rangle\chi_3(x)=\chi_3(x)=-1=\chi_5(x)\\
    \chi_{3}^\sigma(y)
    &= \langle y,x\rangle\chi_3(z')=-\chi_3(y)=+1=\chi_5(y)
  \end{align*}
  This shows $\chi_1^\sigma=\chi_1$ and $\chi_3^\sigma=\chi_5$. The
  same calculations prove $\chi_2^\sigma=\chi_2$ and $\chi_4^\sigma=\chi_6$,
  hence the generator $f_\theta:M\rightarrow M$ defines a twisted symmetry
  action of $\Sigma=\Z_2$ on $M$.
  The covering Nichols algebra $\B(\tilde{M})$ over $G$ then has the asserted
  properties.
\end{proof}

\newpage
\subsection{Ramified Cases \texorpdfstring
{$A_{2n-1}\mapsto C_n$}
{A{2n-1} -> Cn}
}\label{sec_RamifiedAB}

The second ramification will be treated more systematically, by completely
reducing it to the unramified case $A_{n-1}\times A_{n-1}\mapsto A_{n-1}$
and
an additional inert node causing an additionally ramified edge.

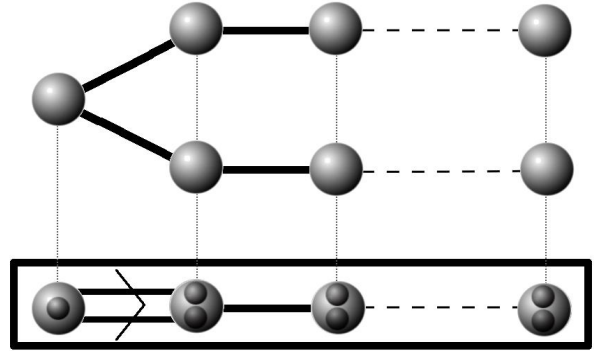
\image{0.26}{DynkinAnBn}

\begin{theorem}\label{thm_RamifiedAB}
  Suppose $G$ a $2$-saturated group
  with $[G,G]=\Z_2$ and $\Gamma:=G/[G,G]$, s.t. 
  \begin{itemize}
    \item $dim_{\F_2}(V)=dim_{\F_2}(G/G^2)=n\geq 3$
    \item $dim_{\F_2}(V^\perp)=dim_{\F_2}(Z(G)/G^2)=1+(n-1\mod 2)$
  \end{itemize}
  Then there exists a suitable $\Gamma$-\ydm of type $A_{2n-1}$
  with an involutory diagram automorphisms. The covering \ydm
  $\tilde{M}$ over $G$ has rank $n$, dimension $2n-1$ and moreover:
  \begin{itemize}
   \item $[G,G]$ acts trivially on $\tilde{M}$, which is
	  diagonal, but $V$  acts faithfully.
   \item $\tilde{M}$ is minimally indecomposable.
   \item $\B(\tilde{M})$ has Hilbert series  
      $$\H(t)=[2]^{2n-1}_{t^1} [2]^{2n-2}_{t^2} [2]^{2n-3}_{t^3}\cdots
	[2]^1_{t^{2n-1}}$$ 
	and is thus especially of dimension 
	$$\dim\left(\B(\tilde{M})\right)=\H(1)=2^{n(2n-1)}
	  =2^{|\Phi^+(A_{2n-1})|}$$
   \item $\tilde{M}$ has the nonabelian Dynkin diagram $C_n$ where short
      roots correspond to conjugacy classes of length $2$ and the
      long root to a central element.
  \end{itemize}
  Exemplary nondiagonal and even faithful Doi twists of
  type $C_3,C_4$ over $\D_4\times \Z_2$ resp. $\D_4\times\Z_2^2$ a given in
  Section \nref{sec_NondiagonalExamples}.
\end{theorem}
\begin{proof}
  As in the ramified case $E_6\mapsto F_4$ above, we use the
  prescribed dimension $1+(n-1 \mod 2)$ nullspace of $V=G/G^2$ to
  decompose $V=\bar{z}\F_2\oplus W$ with $\dim(W^\perp)=n-1\mod 2$ for the
  split nodes and $z\in Z(\Gamma)$ for the inert node.

  Our main goal is to construct a $\Gamma$-\ydm $M$ of dimension $1+2(n-1)$
  and Dynkin diagram $A_{2n-1}$ with the involutory diagram automorphism a
  twisted symmetry. The starting point is the \ydm constructed in the proof of
  Section \nref{sec_Unramified} of dimension $2(n-1)$ and Dynkin diagram
  $A_{n-1}\times A_{n-1}$, numbered $2\ldots 1+2(n-1)$, with an
  involutory twisted symmetry over the subgroup $\Gamma'\subset\Gamma$
  generated  by any lifts of $W$. Denote the leftmost nodes $2,3$ of both copies
  by  $\O_{\bar{g}}^{\chi'},\O_{\bar{g}}^{\chi''}$. We extend all used
characters
  trivially to $\Gamma$ except 
  $$\chi({\bar{z}})=-1\qquad \chi({\bar{g}})=-1 \qquad \chi({\bar{g}}_k)=+1$$
  for all other ${\bar{g}}_k$, which is possible because the images of
  $\bar{g}=\bar{g}_1,\ldots \bar{g}_n$ form a $W$-basis. Note that the
  former
  \ydm had already been proven to be faithful
  over the $\Gamma$-quotient $W$, with $\bar{z}$ now acting trivial on all but
  the new node $M_1$, hence faithfulness of $V$ again holds.

  {\bf First} we have to check that $M$ indeed has decorated diagram
  $A_{1+2(n-1)}$ and hence the asserted Hilbert series by \cite{HS10a} Theorem
  4.5. We've shown that already for the subdiagram $A_{n-1}\times A_{n-1}$, and
  the
  additional node $M_1$ obeys for $k\geq 4$:
  \begin{align*}
      q_{11}
      &=\chi(\bar{z})=-1\\
      q_{12}q_{21}
      &=\chi(\bar{g})\chi'(\bar{z})=(-1)(+1)=-1\\
      q_{13}q_{31}
      &=\chi(\bar{g})\chi''(\bar{z})=(-1)(+1)=-1\\
      q_{1k}q_{k1}
      &=\chi(\bar{g}_k)\chi_k(\bar{z})=(+1)(+1)=+1\\
  \end{align*}
  
  {\bf Secondly} we have to extend the established twisted symmetry $f_p$ of
  $A_{n-1}\times  A_{n-1}$ by $f_p(x_1):=x_1$, which is possible by $z$'s
  centrality in $G$:
  \begin{align*}
    \chi^\sigma(\bar{h})
    &= \sigma(\bar{z},\bar{h})\sigma_p^{-1}(\bar{h},\bar{z})\chi(\bar{h})\\
    &= \langle \bar{h},\bar{z}\rangle\chi(\bar{h})=\chi(\bar{h})
  \end{align*}
  The covering \ydm over $G$ then has the asserted properties.
\end{proof}

\subsection{Disconnected Diagrams}\label{sec_Disconnected}

So far we have constructed covering Nichols algebras $\B(\tilde{M})$ with
connected Dynkin diagram over $2$-saturated groups $G$. We will now show, how
disconnected diagrams can be realized, regardless of the technical assumption
of $G$ to be $2$-saturated. As a corollary, we will note that every
group
$G$ with $[G,G]=\Z_2$ and no restrictions on 2-rank and 2-center admits
(possibly disconnected) finite-dimensional indecomposable Nichols algebras
$\B(\tilde{M})$. Note that in the next Lemma we actually construct a whole
family of covering Nichols algebras for arbitrary 2-rank and 2-center (without
claiming these are all), but for the existence corollary, very simple choices
suffice.

\begin{lemma}\label{lm_Disconnected}
  Let $G$ be a $2$-saturated group with $[G,G]\cong \Z_2$ and 2-rank
  and 2-center 
  \begin{align*}
  n&:=\dim(V)=dim_{\F_2}(G/G^2)\\
  k&:=\dim(V^\perp)=dim_{\F_2}(Z(G)/G^2)
  \end{align*}
  For every numerical decomposition $(n,k)=\sum_i(n_i,k_i)+(0,k_0)$, where all
  $(n_i,k_i)$ appear as 2-rank and 2-center in the list of Theorem
  \nref{thm_AllZpOrbifolds}, we can construct a minimally indecomposable
  covering Nichols algebra over $G$ with connected ramified or
  unramified components as  prescribed by Theorem \nref{thm_AllZpOrbifolds} for
  $(n_i,k_i)$, as well as an  additional inert part $(0,k_0)$, that is an
  arbitrary finite dimensional  indecomposable  Nichols algebra $\B(M^{(0)})$
  over $\Z_2^{k_0}$ from  \cite{Heck09}.
\end{lemma}
\begin{proof}
  Consider again $V=G/G^2$ as a symplectic $\F_2$-vector space. The
  assumed decomposition 
  $$(\dim(V),\dim(V^\perp))=\sum_i(n_i,k_i)+(0,k_0)$$
  implies an
  orthogonal decomposition of $V$ as symplectic vector space into 
  $$V\cong \bigoplus_{i}^\perp V_i \oplus^\perp V_0$$
  where $(n_i,k_i)=(\dim(V_i),\dim(V_i^\perp))$ and $k_0=\dim(V_0)$. Apply the
  constructions of Sections \nref{sec_Unramified}-\nref{sec_RamifiedAB} that
  yield \ydms $M^{(i)}$ over $\Gamma=G/[G,G]$ that factorize over
  $\Gamma\rightarrow V$ and where $\B(M^{(i)})$ having a connected Dynkin 
  diagram of the respective type. Then consider 
  $$M:=\bigoplus_i M^{(i)}\oplus M^{(0)}$$
  where the action of $V_i$ on $M^{(j)}$ for $i\neq j$ is trivial and $M^{(0)}$
  is the assumed \ydm over the abelain group $V_0=\Z_2^{k_0}$.
  Hence $c_{M^{(i)}M^{(j)}}c_{M^{(j)}M^{(i)}}=id$, thus we have trivial adjoint
  action $\ad{\B(M^{(i)})}{\B(M^{(j)})}=0$ and the multiplication in $\B(M)$
  yields an isomorphism of vector spaces
  $$\B(M)\cong \bigotimes_i \B(M^{(i)})\otimes \B(M^{(0)})$$
  and the Dynkin diagram of $\B(M)$ consists of mutually disconnected
  components, each of the respective type of $\B(M^{(i)})$ and $\B(M^{(0)})$.

  Take $f_\theta$ the sum of the twisted symmetries employed in the
  construction of each $M^{(i)}$ and trivial on $M^{(0)}$. Then the covering
  Nichols algebra  $\tilde{M}$ has as Dynkin diagram the mutually disconnected
  Dynkin diagrams of each $\B(\tilde{M}^{(i)})$ as they follow from the
  respective construction and a disconnected inert part with the given
  $\B(M^{(0)})$.	
\end{proof}

\begin{corollary}\label{cor_Nonsaturated}
Let the group beof the form  $G=G_{ab}\times G_{sat}$
with $G_{ab}$ abelian and $G_{sat}$ a $2$-saturated group with
$[G_{sat},G_{sat}]\cong\Z_2$. Note that such a decomposition may not be unique.
Suppose $M_{ab}$ a minimally indecomposable $G_{ab}$-\ydm with
finite-dimensional Nichols algebra $\B(M_{ab})$. Suppose
further over the $2$-saturated group $G_{sat}$ an indecomposable
finite-dimensional Nichols algebra $\B(M_{sat})$ from the list in this article,
i.e. Theorem \nref{thm_AllZpOrbifolds} for connected resp. Lemma
\nref{lm_Disconnected} for disconnected diagrams. Then, the
indecomposable $G$-Nichols algebra
$\B(M_{ab}\oplus M_{sat})$ is as a vector space isomorphic to $
\B(M_{ab})\otimes \B(M_{sat})$, hence 
finite-dimensional,  and has as Dynkin diagram a disjoint union
of the diagrams of $\B(M_{ab}),\B(M_{sat})$.
\end{corollary}

\begin{example}
  Every group $G$ with $[G,G]\cong\Z_2$ is nilpotent and can thus be written as
  a product  $G=G_{odd}\times G_2$, where $G_{odd}$ has odd order and is hence
  abelian,  while $G_2$ is a $2$-group with $[G_2,G_2]\cong \Z_2$ and
  especially $2$-saturated. By the previous lemma we may obtain a 
  finite-dimensional indecomposable Nichols algebra $B(M)$  over $G$ by joining 
  a finite-dimensional indecomposable Nichols algebra over the abelian group 
  $G_{ab}:=G_{odd}$ with a finite-dimensional indecomposable Nichols algebra 
  constructed by Lemma \nref{lm_Disconnected} over $G_{sat}:=G_2$.
\end{example}
Let us come to an explicit easy and generic decomposition:
\begin{example}
  Let thus be $(n,k)$ the type of the symplectic $\F_2$-vector space
  $V=G_2/G_2^2$  as in Lemma \nref{lm_Disconnected}. Applying the lemma,
  to the particular decomposition $(n,0)+(0,k)$ yields a finite-dimensional
  indecomposable covering Nichols algebra $\B(M_2)$, which has as Dynkin diagram
  a disjoint union of an unramified $A_{2n}$ and an arbitrary inert
  finite-dimensional indecomposable Nichols algebra over the abelian group
  $\Z_2^k$, such as a disjoint union of $k$ diagrams of type $A_1$
\end{example}

\begin{corollary}\label{cor_Montgomery}
  Every group $G$ with $[G,G]\cong\Z_2$ admits a
  finite-dimensional indecomposable Nichols algebra $\B(\tilde{M})$. This
  answers for groups of this class positively a question raised by Susan
  Montgomery in 1995 \cite{Mont95}\cite{AS02}: There exist a
  finite-dimensional indecomposable pointed Hopf algebra with coradical
  $\k[G]$, namely the bosonization $\k[G]\#\B(M)$ with $\B(M)$ constructed in
  the previous example, or others constructed by Lemma \nref{lm_Disconnected}.
\end{corollary}

\section{Explicit Examples Of Nondiagonal Nichols
Algebras}\label{sec_NondiagonalExamples}

The covering Nichols algebras $\B(\tilde{M})$ over nonabelian groups $G$
constructed in this
article are by construction non-faithful, because the $G$-action is the
pullback of the action of the quotient $\Gamma$. Especially the commutators
$[G,G]$ act trivially, so the braiding of $\tilde{M}$ is still diagonal.\\

However, over $G$ there may exist Doi twists $\B(\tilde{M}_\eta)$ by a
$G$-group-2-cocycle $\eta\in Z^2(G,\k^\times)$, such that
the action of the subgroup $\Sigma^*=[G,G]\cong\Z_2$ on $\tilde{M}_\eta$ is
nontrivial. Then $\tilde{M}$ has a nondiagonal braiding, and it even may be
faithful, depending on the precise action of the other central elements.\\

In the following we shall derive a criterion for the existence of such
nondiagonal twistings and give a list of examples for covering Nichols algebras
for Rank $2,3,4$.

\subsection{Doi Twists And Matsuomots Spectral Sequence}

	We already noted in Remark \nref{rem_twistedAction}, that a Doi twist of
	the Nichols algebra produces the following twisted action on the
	twisted \ydm $\tilde{M}_\eta$:
	$$a._\eta v_h=\eta(aga^{-1},a)\eta^{-1}(a,g)a.v_h$$
	Hence the central subgroup $\Sigma^*\subset G$ with trivial action on
	$\tilde{M}$ acts on $\tilde{M}_\eta$ 
	by multiplication with the scalar 
	$$\gamma(\eta)(a,g):=\eta(g,a)\eta^{-1}(a,g) \qquad
	  \gamma(\eta)\in\Sigma^*\otimes G$$
	This expression appears
	already in literature on group cohomology, namely in {Matsumoto's
	extension} \cite{IM64} for central group extensions of the
	general Lyndon-Hochschild-Serre spectral sequence\index{Matsumoto's
	spectral sequence}:
	
	$$1\rightarrow \Gamma^*\rightarrow G^*\rightarrow \Sigma\rightarrow
	  H^2(\Gamma,\k^\times)\rightarrow H^2(G,\k^\times)_\Sigma
	  \stackrel{\gamma}{\rightarrow} \Sigma^*\otimes G$$
	$$\qquad$$

	Here, $H^2(G,\k^\times)_\Sigma$ denotes the kernel of the restriction
	map to $\Sigma^*$ and the map $\gamma$ yields as expected a
	{bimultiplicative pairing} that exactly matches the expression above!

	\begin{theorem}\label{thm_Faithful}
	  Let $\Sigma^*=\Z_p\rightarrow G\rightarrow \Gamma$ be a
	  stem-extension, $M$ a $\Gamma$-\ydm with finite-dimensional
	  Nichols algebra $\B(M)$ and $\tilde{M}$ be the covering
	  $G$-\ydmP. If we assume that the following holds
	  $$p\frac{|H^2(G,\k^\times)|}{|H^2(\Gamma,\k^\times)|}>1$$
	  then there exists a group-2-cocycle $\eta\in Z^2(G,\k^\times)$,
	  such that the Doi twist $\tilde{M}_\eta$ has nontrivial action of
	  $\Sigma^*$ and is hence nondiagonal.
	\end{theorem}
	\begin{proof}
	  We use Matsumoto's sequence to
	  enumerate the number of different $\Sigma^*$-actions 
	  $|\Im(\gamma)|$, that can be achieved by Doi twisting. Note that
	\begin{itemize}
	  \item For stem-extensions we have $G^*=\Gamma^*$, so the first terms
	    disappear.
	  \item In our case $\Sigma^*=\Z_p$ we have
	    $H^2(\Sigma^*,\k^\times)=1$, so the restriction kernel is all
	    $H^2(G,\k^\times)_\Sigma=H^2(G,\k^\times)$ 
	\end{itemize}
	so Matsumoto's sequence takes the following form
        $$1\rightarrow \Sigma\rightarrow
	  H^2(\Gamma,\k^\times)\rightarrow H^2(G,\k^\times)
	  \stackrel{\gamma}{\rightarrow} \Sigma^*\otimes G$$
	In particular, counting elements shows the claim: 
	$$|\Im(\gamma)|=|H^2(G,\k^\times)|\cdot
	  |H^2(\Gamma,\k^\times)|^{-1}\cdot |\Z_p|>1$$
	\end{proof}

\begin{remark}\label{rem_reconstruction2}
  This approach has also classificatory value in special cases: In \cite{Len12}
  sections 7.2 -- 7.4 we prove for several exemplary groups $G$ of order $16$
  and
  $32$, that these Doi twists already exhaust all $\Sigma^*$-actions, that
  are possible on a $G$-\ydm with finite dimensional Nichols algebra by
  \cite{HS10a}. Thereby, {\bf all}
  finite-dimensional Nichols algebras over $G$ are Doi twists of covering
  Nichols algebras.\\
  Especially for the last  cases in \cite{Len12} section 7.4, having $G=\Z_2^2$
  and a certain  commutator structure, there is no possible covering Nichols
  algebra and this disproves existence of finite-dimensional
  link-indecomposable Nichols algebras over these groups at all.\\
\end{remark}

\subsection{Examples Of Rank 2}
A symplectic vector space of rank $2=2n+k$ can be of the following two types:\\

Type
$(n,k)=(0,2)$ means that $\bar{G},G$ are abelian.\\

Type $(n,k)=(1,0)$ induces Nichols algebras over the following type of 
group:
%Rank 2, $|G|=16$
%$$\Z\to\bar{G}\to\Z\times \Z\qquad \bar{G}=\Gamma_2$$
$$\Z_2\to G \to\Z_2\times \Z_2\qquad [G,G]=\Z_2\qquad Z(G)=G^2$$
Take as example $G=\D_4$, then $H^2(G,\k^\times)=\Z_2=H^2(\Gamma,\k^\times)$
and hence by Theorem \nref{thm_Faithful} this 2-cocyle causes a nondiagonal Doi
twist (the other stem extension $\Q_8$ has not enough cohomology for
nondiagonal twists).\\
 Recall the generators
%$a,b,\nu\in \Gamma_2\subset \bar{G}$ resp. 
$g,h,\epsilon\in \D_4$ with $gh=\epsilon hg, g^2=\epsilon,h^2=\epsilon^2=1$.\\

A symplectic vector space of type $(1,0)$ admits by Theorem
\nref{thm_SymplecticRootsystem} a symplectic root system of type $A_2$, hence we
obtained in Section \nref{sec_Unramified} an unramified covering Nichols
algebra $A_2\times A_2 \mapsto A_2$. This yields the well known example,
discussed in Section \nref{sec_D4}:
\begin{example}[Type $A_2$, see \cite{MS00} Example 6.5]
$$M=\O_{[h]}^\chi\oplus \O_{[gh]}^\phi$$
Consider the diagonalizable  Yetter-Drinfel'd module $M$ with 
$$\chi(h)=\phi(gh)=-1 \qquad \chi(\epsilon)=\phi(\epsilon)=1$$
as well the nondiagonal (even faithful) Doi twist $M_\eta$ with 
$$\chi(h)=\phi(gh)=-1 \qquad \chi(\epsilon)=\phi(\epsilon)=-1$$
Both Nichols algebras $\B(M),\B(M_\eta)$ are standard of type $A_2$, but
possess a finer PBW-basis of type $A_2\times A_2$. Hilbert series and dimension
are hence:
$$\H(t)=\left([2]_t^2[2]_{t^2}\right)^2 
\qquad\dim=2^{6}=64$$ 
\end{example}

\begin{remark}\label{rem_reflection}
 Here as well as in the following examples, a reflection $R_i$ (see
 \cite{AHS10}) turns the Yetter-Drinfel'd module $M$ into e.g.
  $$R_1M=\O_{[h]}^{\chi'}\oplus \O_{[g]}^{\phi'}$$ 
  where now the generator $g$ has order $4$ (see \cite{Len12} Sec. 5.3 and 8).
  This corresponds to the choice of a different $2$-cocycle $\sigma'\in
  Z^2(\Z^2,\k^\times)$ in the same cohomology class then $\sigma$, especially
  the induced symplectic forms $\sigma(g,h)\sigma^{-1}(h,g)$ coincides.
  Equivalently, for a given $2$-cocycle it corresponds to the choice of a
  different (but isometric) symplectic root system $(\overline{h},\overline{g})$
instead  of $(\overline{h},\overline{gh})$. \\

  Note that we can uniformly write the diagonal and nondiagonal examples and
  all reflections by giving the following character relations and the
  symplectic root system directly, yielding a class of \ydms
  as in \cite{HV13}:
  $$M=\O_{[v]}^\chi\oplus \O_{[w]}^\phi
  \qquad\bar{v}\not\perp\bar{w}
  \qquad \chi(v)=\phi(w)=-1 \qquad \chi(\epsilon)\phi(\epsilon)=1
  $$
\end{remark}

\subsection{Examples Of Rank 3}
A symplectic vector space of rank $3=2n+k$ can be of the following two types: \\

Type $(n,k)=(0,3)$ means that $\bar{G},G$ are abelian.\\

Type $(n,k)=(1,1)$ induces Nichols algebras over the following type of group:
%Rank 3, $|G|=16$
%$$\Z\to \bar{G}\to\Z^3\qquad \bar{G}=\Gamma_2\times\Z$$
$$\Z_2\to G\to\Z_2^3\qquad [G,G]=\Z_2\qquad Z(G)/G^2=\Z_2$$
Take as example for this type $G=\D_4\times \Z_2$, then by K\"uneth's formula 
$H^2(G,\k^\times)=\Z_2^3=H^2(\Gamma,\k^\times)$.
Hence by Theorem \nref{thm_Faithful} there is a 2-cocyle that causes a
nondiagonal Doi twist. Recall the $\D_4$-generators
%$a,b,\nu\in \Gamma_2\subset \bar{G}$ resp. 
$g,h,\epsilon$ and add a
central generator 
%$\bar{z}\in\bar{G}$ resp. 
$z$.

A symplectic vector space of type $(1,1)$ admits by Theorem
\nref{thm_SymplecticRootsystem} symplectic root systems of type $A_3$ or
$A_2\times A_1$, hence we get the following Nichols algebras:
\begin{itemize}
\item From the symplectic root system of type $A_3$ we obtained in Section
\nref{sec_Unramified} an unramified covering Nichols algebra $A_3\times A_3
\mapsto A_3$ as follows:
\begin{example}[Type $A_3$]
$$M=\O_{[h]}^\chi \oplus \O_{[gh]}^\phi\oplus \O_{[zh]}^\psi$$
Consider the diagonalizable  Yetter-Drinfel'd modules $M$ with 
$$\chi(h)=\phi(gh)=\psi(zh)=-1 \qquad
\chi(\epsilon)=\phi(\epsilon)=\psi(\epsilon)=1$$
$$\chi(zh)\psi(h)=1$$
as well as nondiagonal (for some choices faithful) Doi
twists $M_\eta$ with 
$$\chi(h)=\phi(gh)=\psi(zh)=-1 \qquad
\chi(\epsilon)=\phi(\epsilon)=\psi(\epsilon)=-1$$
$$\chi(zh)\psi(h)=1$$
All Nichols algebras $\B(M),\B(M_\eta)$ are standard of type $A_3$, but
possess a finer PBW-basis of type $A_3\times A_3$. Hilbert series and dimension
are hence:
$$\H(t)=\left([2]_t^3[2]_{t^2}^2[2]_{t^3}\right)^2
\qquad \dim=2^{12}=4,096$$
\end{example}

\item From the symplectic root system of type $A_2\times A_1$ we obtained in
Section \nref{sec_RamifiedAB} a ramified covering Nichols algebra $A_5\mapsto
C_3$ as follows:
\begin{example}[Type $C_3$]
$$M=\O_{[h]}^\chi \oplus \O_{[gh]}^\phi\oplus \O_{\{z\}}^\psi$$
Consider the diagonalizable  Yetter-Drinfel'd modules $M$ with 
$$\chi(h)=\phi(gh)=\psi(z)=-1 \qquad
\chi(\epsilon)=\phi(\epsilon)=\psi(\epsilon)=1$$
$$\chi(z)\psi(h)=1\qquad \phi(z)\psi(gh)=-1$$
as well as nondiagonal (for some choices faithful) Doi
twists $M_\eta$ with 
$$\chi(h)=\phi(gh)=\psi(z)=-1 \qquad
\chi(\epsilon)=\phi(\epsilon)=-1\qquad \psi(\epsilon)=1$$
$$\chi(z)\psi(h)=1\qquad \phi(z)\chi(gh)=-1$$
All Nichols algebras $\B(M),\B(M_\eta)$ are standard of type $C_3$, but
possess a finer PBW-basis of type $A_5$. Hilbert series and dimension are hence:
$$\H(t)=[2]_t^5[2]_{t^2}^4[2]_{t^3}^3[2]_{t^4}^2[2]_{t^5}
\qquad\dim=2^{15}=32,768$$
\end{example}

\item Finally we have Nichols algebras with disconnected Dynkin diagram
$A_1\times A_2$ containing an abelian support and an $A_2$-Nichols algebra over
$\D_4$.
\end{itemize}

\subsection{Examples Of Rank 4}
A symplectic vector space of rank $4=2n+k$ can be of the following three
types:\\

Type
$(n,k)=(0,4)$ means that $\bar{G},G$ are abelian. \\

Type $(n,k)=(2,0)$
admits a symplectic root system of type $A_4$. The respective Nichols
algebra over $G=\D_4\ast\D_4$ has Hilbert series 
$$H(t)=\left([2]_t^4[2]_{t^2}^3[2]_{t^3}^2[2]_{t^4}\right)^2
\qquad\dim=2^{20}=1,048,576$$
It is thoroughly discussed in Example \nref{sec_exampleA4}, but over groups of
order $16$ and $32$ we were not able to find a nondiagonal Doi twist (is there
none at all?).\\

Type
$(n,k)=(1,2)$ induces Nichols algebras over the following type
of group:
%Rank 4, $|G|=32$
%$$\Z\to\bar{G}\to\Z^4\qquad \bar{G}=\Gamma_2\times\Z\times \Z$$
$$\Z_2\to G\to\Z_2^4\qquad [G,G]=\Z_2\qquad Z(G)/G^2=\Z_2^2$$
Take as example for this type $G=\D_4\times \Z_2^2$, then by K\"uneth's formula 
$H^2(G,\k^\times)=\Z_2^6=H^2(\Gamma,\k^\times)$.
Hence by Theorem \nref{thm_Faithful} there is a 2-cocyle that causes a
nondiagonal Doi twist. Recall the $\D_4$-generators
%$a,b,\nu\in \Gamma_2\subset \bar{G}$ resp. 
$g,h,\epsilon$ and add central generators 
%$\bar{z},\bar{w}\in\bar{G}$ resp. 
$z,w$.

A symplectic vector space of type $(n,k)=(1,2)$ admits by Theorem
\nref{thm_SymplecticRootsystem} symplectic root systems of type $D_4$ or
$A_1\times A_3$ or $A_1\times A_1\times A_2$, hence we get the
following Nichols algebras:
\begin{itemize}
\item From the symplectic root system of type $D_4$ we obtained in Section
\nref{sec_Unramified} an unramified covering Nichols algebra $D_4\times D_4
\mapsto D_4$ as follows:

\begin{example}[Type $D_4$]
$$M=\O_{[h]}^\chi \oplus \O_{[gh]}^\phi\oplus \O_{[zh]}^\psi\oplus
\O_{[wh]}^\rho$$
Consider the diagonalizable  Yetter-Drinfel'd modules $M$ with 
$$\chi(h)=\phi(gh)=\psi(zh)=\rho(wh)=-1 \qquad
\chi(\epsilon)=\phi(\epsilon)=\psi(\epsilon)=\rho(\epsilon)=1$$
$$\chi(zh)\psi(h)=1\qquad \chi(wh)\rho(h)=1\qquad \psi(wh)\rho(zh)=1$$
as well as nondiagonal (for some choices faithful) Doi twists $M_\eta$
with 
$$\chi(h)=\phi(gh)=\psi(zh)=\rho(wh)=-1 \qquad
\chi(\epsilon)=\phi(\epsilon)=\psi(\epsilon)=\rho(\epsilon)=-1$$
$$\chi(zh)\psi(h)=1\qquad \chi(wh)\rho(h)=1\qquad \psi(wh)\rho(zh)=1$$
All Nichols algebras $\B(M),\B(M_\eta)$ are standard of type $D_4$, but
possess a finer PBW-basis of type $D_4\times D_4$. Hilbert series and dimension
are hence:
$$H(t)=\left([2]_t^4[2]_{t^2}^3[2]_{t^3}^3[2]_{t^4}[2]_{t^5}\right)^2
\qquad\dim=2^{24}=16,777,216$$
\end{example}

\item From the symplectic root system of type $A_1\times A_3$ we obtained in
Section \nref{sec_RamifiedAB} a ramified covering Nichols algebra $A_7\mapsto
C_4$ as follows:

\begin{example}[Type $C_4$]
$$M=\O_{[h]}^\chi \oplus \O_{[gh]}^\phi\oplus \O_{[zh]}^\psi\oplus
\O_{\{w\}}^\rho$$
Consider the diagonalizable  Yetter-Drinfel'd modules $M$ with
$$\chi(h)=\phi(gh)=\psi(zh)=\rho(w)=-1 \qquad
\chi(\epsilon)=\phi(\epsilon)=\psi(\epsilon)=\rho(\epsilon)=1$$
$$\chi(zh)\psi(h)=1\qquad \chi(w)\rho(h)=1\qquad\phi(w)\rho(gh)=1\qquad
\psi(w)\rho(zh)=-1$$
as well as nondiagonal (for some choices faithful) Doi twists $M_\eta$
with 
$$\chi(h)=\phi(gh)=\psi(zh)=\rho(w)=-1 \qquad
\chi(\epsilon)=\phi(\epsilon)=\psi(\epsilon)=-1\qquad \rho(\epsilon)=1$$
$$\chi(zh)\psi(h)=1\qquad \chi(w)\rho(h)=1\qquad\phi(w)\rho(gh)=1\qquad
\psi(w)\rho(zh)=-1$$
All Nichols algebra $\B(M),\B(M_\eta)$ are standard of type $C_4$, but
possess a finer PBW-basis of type $A_7$. Hilbert series and dimension are hence:
$$H(t)=[2]_t^7[2]_{t^2}^6[2]_{t^3}^5[2]_{t^4}^4[2]_{t^5}^3[2]_{t^6}^2[2]_{t^7}$$
$$\dim=2^{28}=268,435,456$$
\end{example}

\item From the symplectic root system of type $A_1\times A_1\times A_2$ we
obtained in Section \nref{sec_RamifiedEF} a ramified covering Nichols algebra
$E_6\mapsto F_4$ as follows:
\begin{example}[Type $F_4$]
$$M=\O_{[h]}^\chi\oplus \O_{[gh]}^\phi\oplus
\O_{\{z\}}^\psi\oplus\O_{\{w\}}^\rho$$
Consider the diagonalizable  Yetter-Drinfel'd modules $M$ with
$$\chi(h)=\phi(gh)=\psi(z)=\rho(w)=-1 \qquad
\chi(\epsilon)=\phi(\epsilon)=\psi(\epsilon)=\rho(\epsilon)=1$$
$$\chi(z)\psi(h)=1\qquad \chi(w)\rho(h)=1\qquad 
\phi(w)\rho(gh)=1$$
$$\phi(z)\psi(gh)=-1\qquad \psi(w)\rho(z)=-1$$
as well as nondiagonal (for some choices faithful) Doi twists $M_\eta$
with 
$$\chi(h)=\phi(hg)=\psi(z)=\rho(w)=-1 \qquad
\chi(\epsilon)=\phi(\epsilon)=-1\qquad \psi(\epsilon)=\rho(\epsilon)=1$$
$$\chi(z)\psi(h)=1\qquad \chi(w)\rho(h)=1\qquad 
\phi(w)\rho(gh)=1$$ 
$$\phi(z)\psi(gh)=-1\qquad \psi(w)\rho(z)=-1$$
All Nichols
algebras $\B(M),\B(M_\eta)$ are standard of type $F_4$, but
possess a finer PBW-basis of type $E_6$. Hilbert series and dimension are hence:
%\begin{narrow}{-1.5cm}{-1cm}
$$\H(t)=[2]^6_{t} [2]^5_{t^2}[2]^5_{t^3}[2]^5_{t^4}
[2]^4_{t^5}[2]^3_{t^6}[2]^3_{t^7} [2]^2_{t^8}
[2]_{t^9}[2]_{t^{10}}[2]_{t^{11}}$$ 
$$\dim=2^{36}=68,719,476,736$$
%\end{narrow}
\end{example}

\item We have 
again several Nichols algebras $X_2^{ab}\cup A_2$ with disconnected Dynkin
diagram and partly abelian support from the symplectic root system
$A_1\times A_1\times A_2$. Here $X_2^{ab}$ may be any diagonal Nichols algebra
of rank $2$. Moreover, we get a ramified $A_1^{ab}\times C_3$. From the
other symplectic root system $A_1\times A_3$ we get $A_1^{ab}\times A_3$.\\
\end{itemize}

\section*{Acknowledgement}
The author wishes to thank  I. Heckenberger, H.-J. Schneider and C. Schweigert
for instructive discussions and helpful comments.\\

The author also wishes to thank the referee for several helpful suggestions,
most notably concerning the structure of Section 6.\\

\bibliographystyle{alpha}
\bibliography{LargeRankNicholsAlgebra}

\end{document}